\documentclass[reqno]{amsart}
\usepackage{amsthm,amsmath,amssymb,oldgerm}
\usepackage{mathrsfs}
\usepackage{url}
\usepackage[a4paper, margin=1in]{geometry}
\usepackage{verbatim}
\usepackage{color}
\usepackage{bbm}
\usepackage[all]{xy}
\usepackage{stmaryrd}
\usepackage{lmodern,graphicx}
\usepackage{calligra} 
\usepackage{hyperref}
\usepackage{xcolor}
\usepackage{enumitem}
\usepackage{stackrel}
\usepackage{mathtools}
\theoremstyle{plain}
\newtheorem{theorem}{Theorem}

\newtheorem{lemma}[theorem]{Lemma}
\newtheorem{corollary}[theorem]{Corollary}

\newtheorem*{unmques}{Question}
\newtheorem*{unmthm}{Theorem}

\theoremstyle{definition}

\theoremstyle{remark}
\newtheorem{remark}{Remark}

\def\Q{\mathbb{Q}}
\def\gp{\mathfrak{p}}
\def\GL{\mathrm{GL}}
\def\Nr{\mathrm{Nm}}

\def\A{\mathbbm A}
\def\R{\mathbb R}

\def\dim{\operatorname{dim}}
\def\b{\mathtt b}
\def\mp{m_{1}m_{2}}
\def\mc{m_{1},m_{2}}
\def\Z{\mathbb Z}

\def\Q{\mathbb Q}
\def\R{\mathbb R}

\def\A{\mathbb A}

\def\Sym{\operatorname{Sym}}

\def\dim{\operatorname{dim}}

\def\tr{\mathrm{Tr}}

\def\O{\operatorname{O}}
\def\o{\operatorname{o}}
\def\log{\operatorname{log}}

\def\gp{\mathfrak p}

\def\SU{\mathrm{SU}}

\begin{document} 
\title{A discrepancy result for Hilbert modular forms}

\author[Baskar Balasubramanyam]{Baskar Balasubramanyam}
\date{\today}
\address{Baskar Balasubramanyam, IISER Pune, Dr Homi Bhabha Road, Pashan, Pune - 411008, Maharashtra, India}
\email{baskar@iiserpune.ac.in}
\author[Jishu Das]{Jishu Das}
 \address{Department of Applied Sciences and Humanities, Mathematics Section, 
 COEP Technological University, Wellesley Road, Shivajinagar, Pune, Maharashtra 411005, India.}
\email{dasj.maths@coeptech.ac.in}
\author[Kaneenika Sinha]{Kaneenika Sinha}
\address{Kaneenika Sinha, IISER Pune, Dr Homi Bhabha Road, Pashan, Pune - 411008, Maharashtra, India}
\email{kaneenika@iiserpune.ac.in}

\begin{abstract}
Let $F $ be a totally real number field and $r=[F :\mathbb{Q}].$ Let $A_k(\mathfrak{N},\omega) $ be the space of holomorphic Hilbert cusp forms with respect to  $K_1(\mathfrak{N})$, of weight $k=(k_1,\dots,k_r)$ such that $k_j>2$ for all $j$, and with central Hecke character $\omega$. For integral ideals $\mathfrak{N}$ and $\mathfrak{n}$ in $F$ such that $( \mathfrak{n}, \mathfrak{N}) = 1$, we study the Petersson trace formula  for the Hecke operator $T_{\mathfrak{n}}$ acting on the space $A_k(\mathfrak{N},\omega)$.  We present asymptotic estimates for the terms of the Petersson formula  as $k_0\rightarrow\infty,$  where $k_0=\min(k_1,\dots,k_r)$. As an application, we obtain a weighted discrepancy bound for the distribution of the eigenvalues of the Hecke operator $T_{\mathfrak{p}}$ (for a fixed prime ideal $\mathfrak{p}$) acting on the space $A_k(\mathfrak{N},1),$ when $F$ has narrow class number $1$,  and the ideal $\mathfrak{N}$ is generated by (rational) integers.  This generalizes a discrepancy result previously obtained by Jung and Sardari \cite{JS} in the context of classical cusp forms.
\end{abstract}

\maketitle

\section{Introduction}

Let $F$ be a totally real number field of degree $r$ over $\Q$. Let $\A = \A_F$ denote the adele ring of $F$. Let $\pi$ be a regular, algebraic, cuspidal, automorphic representation for $\GL_2$ over $\A$. Let $v$ be a non-Archimedean place of $F$ where $\pi$ is unramified. Then, $\pi_v$ is determined by its Satake parameters $\alpha_{\pi} (v)$ and $\beta_{\pi} (v)$. The sum $\kappa_\pi (v) = \alpha_{\pi} (v) + \beta_{\pi} (v)$ is the normalised Hecke eigenvalue of the Hecke operator $T_v$ acting on $\pi_v$. The generalised Ramanujan conjecture states that $|\alpha_{\pi} (v)| = |\beta_{\pi} (v)| = 1$, which is further equivalent to the assertion that $|\kappa_\pi (v)| \leq  2$. This conjecture has been proved in the case when $\pi$ is holomorphic in a series of important results by \cite{Deligne}, \cite{Deligne-Serre}, \cite{Brylinski-Labesse}, \cite{Blasius} and \cite{Nguyen}.  We refer the reader to \cite{Li-2009} for a sequential description of these developments.    

In the case when $\pi$ is not of CM-type, the (generalized) Sato-Tate conjecture states that the Hecke eigenvalues 
$$
\{ \kappa_\pi (v) \mid v \textrm{ where } \pi_v \textrm{ is unramified}\}
$$
are equidistributed in the interval $[-2,2]$ with respect to the measure 
$$d\mu_{\infty}(t) := \frac{1}{\pi}\sqrt{1 - \frac{t^2}{4}}dt.$$
Here the representation $\pi$ is fixed and the primes $v$ vary.  This conjecture is now a theorem by the work of \cite{BGHT} for the case when $F = \Q$, and \cite{BGG} for any totally real field $F$ over $\Q$.  

In a related but different perspective, the ``vertical" Sato-Tate distribution law considers the equidistribution of these Hecke eigenvalues for a fixed prime $v$ with the representations $\pi$ varying. To make this precise, fix $v$ and let $\mathfrak N$ be an integral ideal of $F$ such that $ v  \nmid \mathfrak N$. Let $k = (k_1, \dots, k_r) \in \Z^r$ with $k_i \ge 2$. Let $\Pi_k (\mathfrak N)$ denote the set of cuspidal holomorphic automorphic representations with trivial central characters $\pi$ such that $\pi_{\mathrm{fin}} = \bigotimes_{v \nmid \infty} \pi_v$ has a non-zero $K_0 (\mathfrak N)$-fixed vector and $\pi_\infty = \otimes \pi_{k_i}$, where $\pi_{k_i}$ are the discrete series representations of weight $k_i$. In the language of classical Hilbert modular forms, the first condition is equivalent to saying that $\pi$ comes from a Hilbert modular form of conductor dividing $\mathfrak{N}$. In particular, $\pi_v$ is unramified for all $v \nmid \mathfrak N$. The second condition is equivalent to saying that $\pi$ comes from a holomorphic Hilbert cusp form of weight $k = (k_1, \dots, k_r)$. 

Let $\mathfrak{N_i}$ be a sequence of ideals (coprime to $v$) and $\mathfrak{N_i} \to \infty$ (i.e., the norms of these ideals go to infinity). The vertical Sato-Tate distribution theorem in this case states that the family of multisets
$$
S_i = \{ \kappa_\pi (v) \mid \pi \in \Pi_k (\mathfrak{N}_i) \}
$$
is equidistributed with respect to the measure
$$
d \mu_v (x) = \frac{\mathrm{Nm}(v) + 1}{\left({\mathrm{Nm}(v)}^{1/2} + {\mathrm{Nm}(v)}^{-1/2}\right)^2 - x^2 } d \mu_\infty (x).
$$
The vertical perspective was first introduced in the work of \cite{Sarnak} and \cite{Serre} in the classical setting of $F = \Q$.  While Sarnak \cite{Sarnak} addressed the case of Maass cusp forms, Serre \cite{Serre} studied the vertical framework in the case of holomorphic cusp forms as an application of general equidistribution principles.  
In the general context of totally real number fields $F$ with $[F:\Q] \geq 2$, the study of the equidistribution of the families $S_i$ was initiated in \cite{Li-2009}.  

The  equidistribution theorems in the horizontal and vertical aspects are obtained through different techniques.  The horizontal Sato-Tate distribution theorem for a fixed non-CM $\pi$ is obtained by viewing the Hecke eigenvalues $\{\kappa_{\pi}(v)\}$ as coefficients of the $L$-function corresponding to $\pi$, and showing that the symmetric power $L$-functions $L(s, \Sym^n\pi)$ for all $n \geq 1$ are potentially cuspidal automorphic.  On the other hand, in the vertical context, the families $S_i$ are viewed as the eigenvalues of the Hecke operator $T_{\mathfrak p}$ acting on $\Pi_k (\mathfrak N)$, where $\mathfrak p$ is the prime ideal corresponding to the non-Archimedean valuation $v$.   These families are studied with the help of a trace formula for the traces of the Hecke operators $T_{\mathfrak p^n}$ for all powers $n \geq 1$.  For $F = \Q$, the primary tool is the Eichler-Selberg trace formula.  We refer the reader to \cite{KL-traces} for a detailed description of the history and development of this formula, as well as a proof using the modern perspective of the Arthur-Selberg trace formula for $\GL_2$ over $\A_{\Q}$.  The adelic perspective forms the basis for a suitable generalization of the Eichler-Selberg trace formula to Hilbert modular forms \cite{Li-2009}. 

For the purpose of motivating the questions addressed in this article, we start with a review of some equidistribution results in the classical context of modular forms.
Let $S_k(N)$ denote the space of cusp forms of even integer weight $k$ with respect to $\Gamma_0(N)$.  Let $\mathcal{F}_k(N)$ be an orthonormal basis of $S_k(N)$ consisting of joint eigenfunctions of the Hecke operators $T_n$ with $(n,N)=1$. For  $f\in S_k(N)$,
the Fourier expansion of $f$ at the cusp $\infty$ is given by $$f(z)=\sum_{n=1}^{\infty} a_f(n)n^{\frac{k-1}{2}}  e^{2\pi i nz}.$$  We denote $\kappa_f(n)$ to be the $n$-th normalised Hecke eigenvalue of $f$.  That is, $\kappa_f(n)$ denotes the eigenvalue of the normalised Hecke operator $T_n/n^{(k-1)/2}$ acting on $f$.  Thus,  
$$a_f(n)=a_f(1)\kappa_f(n),\,(n,N) = 1.$$
Let $p$ be a fixed prime number with $\gcd\,(p,N)=1.$  By the Ramanujan-Deligne bound, we know that $\kappa_f(p) \in [-2,2]$.  

We define, on the interval $[-2,2]$, the measures 
$$d\mu_{\infty}(x) := \frac{1}{\pi}\sqrt{1 - \frac{x^2}{4}}dx,$$
and
$$d\mu_p(x):=\frac{p+1}{\pi}\frac{1}{(p^{1/2}+p^{-1/2})^2-x^2}d\mu_{\infty}(x).$$ 
For an interval $I \subset [-2,2]$, and for a fixed prime $p$, denote
$$\mu_{k,N}(I):= \frac{1}{|\mathcal{F}_k(N)|}\sum_{f\in \mathcal{F}_k(N)}\delta_{\kappa_f(p) \in I},$$
$$\mu_{\infty}(I) := \int_I d\mu_{\infty}(x),$$
and
$$\mu_p(I) := \int_I d\mu_p(x).$$

Serre \cite{Serre} proved that for a fixed prime $p$, the families of eigenvalues
$$\{\kappa_f(p) \mid f \in \mathcal F_k(N)\}$$
are equidistributed in the interval $[-2,2]$ with respect to the measure $d\mu_p(x)$
as $k+N\rightarrow \infty$ with $k$ even and $\gcd(p,N)=1.$  That is, for any interval $I \subset [-2,2]$,
\begin{equation}\label{v-d-Hecke}
\lim_{k+N\rightarrow \infty \atop {(p,N) = 1 \atop{k \text{ even}}}} \mu_{k,N}(I) = \mu_p(I).
\end{equation}
The above asymptotic for the case $N=1$ was also proved by Conrey, Duke and Farmer \cite{CDF}.  

Let $S_k(N)^*$ be the subspace of primitive cusp forms in $S_k(N)$.  Let $T^*_n$ be the restriction of Hecke operator $T_n$ on the subspace $S_k(N)^*$ of $S_k(N)$.  Let 
$\mathcal{F}_k(N)^*$ be the orthogonal basis of $S_k(N)^*$ consisting of joint eigenfunctions of the Hecke operators $T_n^*$, $(n,N) = 1$ such that $a_f(1) = 1$ for all $f \in \mathcal{F}_k(N)^*$.
We denote $\mu_{k,N}^*$ to be the corresponding measure on $[-2,2]$ associated to $T^*_p$.  That is, for an interval $I \subset [-2,2]$,
$$\mu_{k,N}^*(I):= \frac{1}{|\mathcal{F}_k(N)^*|}\sum_{f\in \mathcal{F}_k(N)^*}\delta_{\kappa_f(p) \in I}. $$
Using \eqref{v-d-Hecke}, it is not too difficult to deduce that for a fixed prime $p$, and for an interval $I \subset [-2,2]$,
\begin{equation}\label{v-d-Hecke-new}
\lim_{k+N\rightarrow \infty \atop {(p,N) = 1 \atop{k \text{ even}}}} \mu_{k,N}^*(I) = \mu_p(I).
\end{equation}
The above mentioned equidistribution theorems naturally lead to questions about the discrepancies 
$$D(\mu_{k,N},\mu_p) := \sup\left\{   |\mu_{k,N}(I)-\mu_p(I)|\, \Big|\, I=[a,b]\subset [-2,2] \Big. \right\}$$
and
$$D(\mu_{k,N}^*,\mu_p) := \sup\left\{   |\mu^*_{k,N}(I)-\mu_p(I)|\, \Big|\, I=[a,b]\subset [-2,2] \Big. \right\}.$$
These have been well investigated, along with applications (see \cite{GJS}, \cite{Royer}, \cite{Golubeva}, \cite{MSeffective}, \cite{MS2}, \cite{SZ}, \cite{JS}).  

In \cite{GJS}, Gamburd, Jakobson and Sarnak obtain the bound 
$$D(\mu_{k,2},\mu_p) = \O\left(\frac{1}{\log k}\right)$$
by a study of the eigenvalues of $\widehat{z}(\pi_k)$, as $k \to \infty$, where $z$ is an element in the group ring $\R[\SU(2])$, and $\pi_k$ is the irreducible representation of $\SU(2)$ of dimension $k+1$ (see \cite[Theorem 1.3]{GJS}). 

For $N=1$, Golubeva \cite{Golubeva} obtains the bound 
$$D(\mu_{k,1},\mu_p) = \O_{\epsilon}\left(\frac{1}{(\log k)^{1 - \epsilon}}\right)$$
for any $\epsilon > 0$.

This can be sharpened into the bound
\begin{equation}\label{error-Hecke}
D(\mu_{k,N},\mu_p)=\O\left(\frac{1}{\log kN} \right)
\end{equation}
\cite{MSeffective}, which holds  for all even positive integers $k$ and for all $N \geq 1$ such that $(p,N) = 1.$
A similar upper bound \cite{MS2} can also be obtained for $D(\mu_{k,N}^*,\mu_p)$:
\begin{equation}\label{error-Hecke-nf}
D(\mu_{k,N}^*,\mu_p)=\O\Big(\frac{1}{\log kN}\Big).
\end{equation}
The primary tool used in \cite{Golubeva} and  \cite{MSeffective} is the Eichler-Selberg trace formula, combined with an approximation of the characteristic function of an interval by certain smooth test functions.  Indeed, for any interval $I \subset [-2,2]$, and any positive integer $M \geq 1$, we have an Erd\"{o}s-Tur\'{a}n type of inequality: 
\begin{equation}\label{Erdos-Turan}
 |\mu_{k,N}(I)-\mu_p(I)| \leq \frac{b_1\dim(S_k(N))}{M+1} + \sum_{m=1}^M {b_{m,I} G(m)},
 \end{equation}
where $b_1$ and $b_{m,I}$ are positive constants derived from the above test functions, and $G(m)$ depends on the sum $\sum_{f \in \mathcal F_k(N)} \kappa_f(p^m)$, which is $\mathrm{Tr}\left(\frac{T_{p^m}}{{p^m}^{(k-1)/2}}\right)$.  We choose a value of $M$ to obtain an optimal bound for the right hand side of \eqref{Erdos-Turan}.  This technique works for all positive integers $N$ (coprime to $p$), and gives us the discrepancy bounds
\eqref{error-Hecke} and \eqref{error-Hecke-nf}, where the implied constants are absolute and effectively computable.  However, this method faces a barrier: the Eichler-Selberg trace-formula estimates for the sum $\sum_{f \in \mathcal F_k(N)} \kappa_f(p^m)$, when applied to \eqref{Erdos-Turan}, generate a non-trivial bound only for choices of $M$ such that $p^M \ll (kN)^c$ for some $0 < c < 2/3 $.

The application of the Petersson trace formula for the study of discrepancies has been explored in \cite{JS} and \cite{SZ}.   This formula evaluates the harmonic sums 
$$\sum_{f \in \mathcal F_k(N)} \omega_f \kappa_f(n_1)\overline{\kappa_f(n_2)},$$
where 
\begin{equation}\label{omega-f}
\,\omega_f := \frac{\Gamma(k-1)}{(4\pi)^{k-1}} |a_f(1)|^2,
\end{equation}
and $n_1$ and $n_2$ are positive integers coprime to $N$.  For an application to the current context, we choose $n_1 = p^m$ and $n_2 = 1$.  The perspective of \cite{SZ} is to extend the range of $m$ for a sharper estimation of the trace $\sum_{f \in \mathcal F_k(N)} \kappa_f(p^m)$ through the Petersson trace formula.  A weight-removal technique is applied to the above sum by interpreting the weight $\omega_f$ in terms of the special value of the symmetric square $L$-function $L(s,\Sym^2f)$ at $s=1$, and by a suitable approximation for $L(s,\Sym^2f)$.  This approach was followed in \cite{SZ} to derive bounds for the discrepancies $D(\mu_{k,N},\mu_p)$ and $D(\mu^*_{k,N},\mu_p)$ in the case when $k = 2$.  In the context of Hecke-Maass forms, the use of the analogous Kuznetsov trace formula and weight removal to obtain the requisite discrepancy bounds have been explored in \cite{LW} and \cite{TW}.

 \begin{unmques}Is it possible to improve the logarithmic saving in \eqref{error-Hecke} to a power saving in the bound for $D(\mu_{k,N},\mu_p)$ (and likewise, for $D(\mu_{k,N}^*,\mu_p)$)? That is, does there exist $0 < \theta <1$ such that 
 $$D(\mu_{k,N},\mu_p) = \O\left(\frac{1}{(kN)^{\theta}}\right)?$$
 \end{unmques}
 Related to the above question is that of obtaining $\Omega$-type estimates for these discrepancies.  In what follows, we restrict our attention to the spaces $S_k(N)^*$.
\begin{unmques}
Can we find a function $E(k,N)$ for positive integers $(k, N)$ such that 
$$D(\mu_{k,N}^*,\mu_p) = \Omega (E(k,N))\text{ as }k + N \to \infty ?$$
That is, can we find a sequence $(k_{n},N_{n})_{n \geq 1}$ such that $k_{n}$ is even, $(p,N_{n}) = 1$ and 
$$D(\mu_{k_{n},N_{n}}^*,\mu_p) \gg E\left(k_{n},N_{n}\right) \text{ as }k_{n} + N_{n} \to \infty ?$$
\end{unmques}

The above question was addressed for $N=2$ by Gamburd, Jakobson and Sarnak \cite{GJS}, and for all squarefree levels $N$ by Jung and Sardari \cite{JS}.  In \cite{GJS}, it is shown that there exists a sequence of even integers $k_n \to \infty$ such that
\begin{equation}\label{Discrepancy-unw2}
D(\mu_{k_n,2}^*,\mu_p)\gg \frac{1}{k_n^{\frac{1}{2}}\log^2 k_n}.
\end{equation} 
Jung and Sardari \cite[Theorem 1.6]{JS} generalize the above result to any fixed squarefree level $N$ with an improved exponent for $k_n$.  That is, given a fixed squarefree level $N$, they obtain a sequence of weights $k_n$ with $k_n \rightarrow \infty $ such that
\begin{equation}\label{Discrepancy-unw}
D(\mu_{k_n,N}^*,\mu_p)\gg \frac{1}{k_n^{\frac{1}{3}}\log^2 k_n}.
\end{equation} 
This tells us that if the answer to the first question above is affirmative, then $\theta < 1/3$.

Define
$$\mathcal H_k(N)^* := \sum_{f \in \mathcal F_k(N)^*} \omega_f,$$
with $\omega_f$ as in \eqref{omega-f}.

We define the Borel measure $\nu_{k,N}^*$ on $[-2,2]$ as follows: for an interval $I \subset [-2,2]$, 
$$\nu_{k,N}^*(I):= \frac{1}{\mathcal H_k(N)^*}\sum_{f \in \mathcal{F}^*_k(N) } \omega_f \,\delta_{\kappa_f(p) \in I}.$$
Using the Petersson trace formula, one can show a weighted variant of \eqref{v-d-Hecke-new}, namely that
\begin{equation}\label{v-d-Hecke-new-weighted}
\lim_{k+N\rightarrow \infty \atop {(p,N) = 1 \atop{k \text{ even}}}} \nu_{k,N}^*(I) = \mu_{\infty}(I).
\end{equation}
We refer the interested reader to \cite{Li}, \cite{OM} and \cite{KL1} for a discussion of weighted distribution theorems for Hecke eigenvalues, other  variants and analogues for Hecke-Maass cusp forms.
 
 In the context of discrepancies, Jung and Sardari obtain, for a fixed squarefree level $N$, a sequence of weights $k_n$ with $k_n \rightarrow \infty $ such that the lower bound 
\begin{equation}\label{Discrepancy}
    D(\nu_{k_n,N}^*,\mu_\infty)\gg \frac{1}{k_n^{\frac{1}{3}}\log^2 k_n}
\end{equation}
holds. The transition from \eqref{Discrepancy} to \eqref{Discrepancy-unw} is made with the help of an explicit asymptotic version of the Petersson trace formula.  Recently, an extension of \eqref{Discrepancy} to any positive integer $N$ which is not divisible by $8$ was obtained in \cite[Theorem 1]{JD}. 

 The discrepancy bound in \eqref{error-Hecke} was extended to the context of Hilbert modular forms in \cite{LLW}.  This gives us a quantitative version of the vertical equidistribution result of \cite{Li-2009} with effective error terms.  The strategy of proof in \cite{LLW} is to use a ``higher-dimensional" variant of the Erd\"{o}s-Turan inequality (see \cite[Section 7]{Li-2009}).  The analogue of the term $G(m)$ in the case of Hilbert modular forms is then estimated with the help of the Arthur-Selberg  trace formula (see \cite[Section 3]{LLW}).
 
However, the extension of the $\Omega$-type bounds in \eqref{Discrepancy-unw} and \eqref{Discrepancy} to Hilbert modular forms has not been addressed yet.  In this article, we initiate a discussion on the extension of \eqref{Discrepancy} to Hilbert modular forms.  We first recall the following weighted equidistribution result of Knightly and Li \cite[Theorem 1.1]{KL}.

\begin{unmthm}[\cite{KL}]
Let $F$ be a totally real number field, and let $m$ be a totally positive element of the inverse different $\mathfrak{d}^{-1}\subset F.$  For a cusp form $\phi $ on $\GL_2(F) \backslash \GL_2({\A}_F)$ with trivial central character, let $W_m^\phi $ denote its $m$-th Fourier coefficient (see \eqref{eqn: Fourier-coefficient}). The weight associated to the cusp form $\phi$ is defined as
$$
w_\phi:=\frac{|W_m^\phi(1)|^2}{\| \phi\|^2},
$$
where $\| \phi\|$ is the Petersson norm of $\phi$.
 
For an integral ideal $\mathfrak{N}$ of $\mathcal{O}$, let $A_k(\mathfrak{N},1)$ denote the space of holomorphic Hilbert cusp forms of weight $ k = (k_1,k_2,\dots,k_r)$ (each $k_i >2$ and even) with respect to the Hecke congruence subgroup $\Gamma_0(\mathfrak{N})$.  

 For integral ideals $\mathfrak{n}$ and $\mathfrak{N}$ such that $(\mathfrak{n},\mathfrak{N}) = 1$, let $T_\mathfrak{n}$ denote the $\mathfrak{n}$-th Hecke operator acting on $A_k(\mathfrak{N},1)$.  Let $\mathcal F_k(\mathfrak{N})$ be a Hecke eigenbasis of $A_k(\mathfrak{N},1)$, that is, a basis of $A_k(\mathfrak{N},1)$ consisting of simultaneous eigenfunctions of the Hecke operators $T_\mathfrak{n}.\,$
 
For each $\phi \in \mathcal F_k(\mathfrak{N})$, let $\lambda^\phi_{\mathfrak{n}}$ be an eigenvalue for the Hecke operator $T_{\mathfrak{n}}$ with eigenvector $\phi$.  Let $\kappa^\phi_{\mathfrak{n}}:=\frac{\lambda^\phi_{\mathfrak{n}}}{\sqrt{\Nr(\mathfrak{n})}}.$
 
Consider a fixed prime ideal $\mathfrak{p}$ not dividing $m\mathfrak{d}$.  For an interval $I \subset [-2,2]$, define
\begin{equation}\label{nu-HF-1}
{\nu}_{k,\mathfrak{N}}(I) := \frac{1}{\mathcal H_k(\mathfrak{N})}\sum_{\phi \in \mathcal F_k(\mathfrak{N})}\omega_{\phi}\delta_{\kappa^\phi_{\mathfrak{p}} \in I},
\end{equation}
where
$$\mathcal H_k(\mathfrak{N}) = \sum_{\phi \in \mathcal F_k(\mathfrak{N})}\omega_{\phi}.$$
Then,
$$
\lim_{\Nr(\mathfrak{N})\rightarrow \infty \atop {(\mathfrak{p}, \mathfrak{N}) = 1}} {\nu}_{k,\mathfrak{N}}(I) = \mu_{\infty}(I).
$$
\end{unmthm}

This leads to questions about the discrepancies $D(\nu_{k,\mathfrak{N}},\mu_{\infty})$.  In the current article,  our main goal is to generalize equation \eqref{Discrepancy} to the context of  Hilbert cusp forms. The proof of  \eqref{Discrepancy} follows from asymptotic estimates for the Petersson trace formula for the Hecke operators $T_n$ acting on the space $S_k(N)^*$ where the level $N$ is square-free, and $n$ lies in a certain range dependent on $k$.  Thus, the first theorem in this article is an asymptotic version of the Petersson trace formula for Hilbert cusp forms.  We review some notation before the statement of the theorem.

\begin{itemize}
\item Let $F^{+}$ denote the set of totally positive elements of $F$. Let $\A$ (or $\A_F$) denote the adele ring of $F$. Since $F$ is fixed throughout this manuscript, we usually omit it from the notation.
\item Let $\mathcal{O}^\times$ denote the unit group of $F$, and let $U$ be a fixed set of representatives for $\mathcal{O}^\times/{\mathcal{O}^\times}^2$.  Note that $U$ is a finite set and $|U| = 2^r$, where $[F:\Q]=r$. 
\item Let $\sigma_1, \dots, \sigma_r$ be the embeddings of $F$ into $\R$ and let $\sigma = (\sigma_1, \dots, \sigma_r) : F \to \R^r$.  
\item Let $\mathfrak{n}$ and $\mathfrak{N}$ be ideals in $\mathcal O$ such that $( \mathfrak{n}, \mathfrak{N}) = 1$.  
\item For a non-Archimedean place $v$ of $F$, let $\mathcal O_v$ denote the ring of integers in the local field $F_v$.  Let $\widehat{\mathcal O} = \prod_{v < \infty}{\mathcal O}_v.$ 
\item For a fractional ideal $\mathfrak a \subset F$, let $\mathfrak a_v$ denote its localization.  We write 
$$\widehat{\mathfrak{a}} = \mathfrak a \widehat{\mathcal O} = \prod_{v < \infty} \mathfrak a_v.$$
Further, let $[\mathfrak a]$ denote the image of $\mathfrak a$ in the ideal class group.
\item Consider the equation $1 = [\mathfrak{b}]^2[\mathfrak{n}]$ in terms of ideal class groups, and let 
$$[\mathfrak{b}_1],\,[\mathfrak{b}_2],\,\dots,\,[\mathfrak{b}_t]$$
be solutions of the above equation.   We choose $\eta_i \in F$ such that $\eta_i$ generates the principal ideal $\mathfrak{b}_i^2\mathfrak{n}$. Also, $\b_i$ is such that $\b_i \Hat{\mathcal{O}}=\Hat{\mathfrak{b}}_i$ for the  integral ideal $\mathfrak{b}_i.$ 
\item  Let $\|x\|$ denote the standard Euclidean norm of $x\in \R^r$.  Let us consider 
\begin{align}\label{delta_i}
\delta_i =\inf\{\|\sigma(s)\|  \,\mid \, s\in \mathfrak{b_i}\mathfrak{N}/\pm, s\neq 0 \},
\end{align}
where each $\mathfrak{b_i}$ is as defined above.  We also denote $\tilde{\delta}_i=\frac{\delta_i}{2\sqrt{r}}$.
\item  
We define
\begin{align}\label{A_i}
A_i=\bigcap_{j=1}^r\{  s\in \mathfrak{b_i}\mathfrak{N}/\pm \,\mid\, |\sigma_j(s)|\leq  2 \delta_i,\ s\neq 0\}.
\end{align} 
For each $i$, $A_i$ is a discrete bounded set in $\R^r$, hence finite. This will contribute to the main term in our estimates below. 
\item We define
$$\gamma_{j}=\max \left \{\sqrt{\sigma_{j}(\eta_i u)} \,\mid \,i =1,\,...\,,t  ,u\in U, \eta_i u\in F^+\right \} .$$   
\end{itemize}

\begin{theorem}\label{Main theorem}
Let $\mathfrak{N}$ and $\mathfrak{n}$ be fixed integral ideals in $F$ such that $( \mathfrak{n}, \mathfrak{N}) = 1$.  Let $\omega:\,F^{\times} \backslash \mathbb A^{\times} \to \mathbb C^{\times}$ be a unitary Hecke character. Let the conductor of $\omega$ divide $\mathfrak{N}$ and $\omega_{\infty_j}(x)=\mathrm{sgn} (x)^{k_j} $ for all $j=1,\dots, r$.   For $k = (k_1,k_2,\dots, k_r)$ with all $k_j >2$, let $A_k(\mathfrak{N},\omega)$ denote the space of Hilbert cusp forms of weight $k$ and character $\omega$ with respect to $K_1(\mathfrak{N})$ (see equation \eqref{K_1-def} for the definition of $K_1(\mathfrak{N})$) and let $\mathcal{F}$ be an orthogonal basis for $A_{k}(\mathfrak{N},\omega)$ consisting of eigenfunctions of the Hecke operator $T_{\mathfrak{n}}$.  Let $A_i$'s and $\tilde{\delta}_i$  be as defined above and $k_0=\min \{ k_j \,  | \, j\leq r\}$. Suppose $m_1,m_2 \in \mathfrak{d}^{-1}_+$  and  satisfy
\begin{equation}\label{m1jm2j}
\frac{2\pi \gamma_j\sqrt{\sigma_{j}(m_1m_2)}}{\tilde{\delta}_i}\in \Big((k_{j}  -1)-(k_{j}  -1)^{\frac{1}{3}},(k_{j}  -1)\Big) \text{ for all }j \leq r.
\end{equation}
Then,  as $k_0\rightarrow  \infty$,
$$ 
\frac{e^{2\pi \tr_{\mathbb{Q}}^F (m_1+m_2)}}{{\psi(\mathfrak{N})}}
 \Bigg[\prod_{j=1}^r  \frac{(k_j-2)!}{(4\pi \sqrt{\sigma_j(m_1m_2)})^{k_j-1}} \Bigg]
\sum_{{\phi} \in \mathcal{F} }\frac{\lambda_\mathfrak{n}^\phi W_{m_1}^\phi(1) \overline{W_{m_2}^\phi(1)}}{\|\phi \|^2}
 $$
$$=\, \Hat{T}(m_1,m_2,\mathfrak{n})\frac{\sqrt{d_F\Nr(\mathfrak{n})}}{\omega_\mathfrak{N}(m_1/s)\omega_{\mathrm{f}}(s)}
+ \sum_{i=1}^t \sum_{u\in U, \eta_i u\in F^+}\sum_{s\in A_i }\Bigg\{ \omega_{\mathrm{f}}(s\b_i^{-1} ) S_{\omega_\mathfrak{N}} (m_1,m_2;\eta_i u \b_i^{-2};s\b_i^{-1})
 $$$$ 
\frac{\sqrt{\Nr(\eta_i u)}}{\Nr(s)}\times  \prod_{j=1}^r\frac{2\pi}{(\sqrt{-1})^{k_j}}J_{k_j-1} \Big(  \frac{4 \pi\sqrt{\sigma_j (\eta_i u m_1m_2 )}}{|\sigma_j(s)|}\Big)
 \Bigg\}+\O_{A_i,u,\eta_i} \left(\sum_{j=1}^r \left( (k_{j}-1)^{-\frac{1}{2}}\prod_{\substack{j'=1\\ j'\neq j}}^r(k_{j'}-1)^{-\frac{1}{3}}\right)\right),$$
for all $j$. 
 A detailed description of all the terms above is provided in \S\ref{sec: Peetersson-trace-formula}.
\end{theorem}
\begin{remark}
Here are some features and comments for the above theorem. 
\begin{itemize}
\item We note that the values of  $m_1$ and $m_2$ are dependent on $k$.  Since the lengths of the intervals $ \left((k_{j}  -1)-(k_{j}  -1)^{\frac{1}{3}},(k_{j}  -1)\right)$ 
go to infinity, we can always choose infinitely many $m_1,m_2$ satisfying \eqref{m1jm2j}, as $k_0 \to \infty$.
\item The main idea of the proof is to use the Petersson trace formula in the Hilbert modular setting, which was first stated in \cite{KL} and is recalled in this article as Theorem \ref{Petersson-Trace-Formula}.  We then apply Lemma \ref{Bessel-estimate} carefully using a combinatorial arithmetic argument to arrive at the estimate. 
\item  Theorem \ref{Main theorem} is an analogue of \cite[Theorem 1.7]{JS} for the space $A_{k}(\mathfrak{N},\omega)$.  See also \cite[Theorem 4(i)]{JD}. 
\end{itemize}
\end{remark}

In our next result, we find a lower bound for the main term in Theorem \ref{Main theorem} under additional assumptions that $F$ has odd narrow class number and the ideal $\mathfrak{b}_1\mathfrak{N}= \tilde{s}\mathcal{O} $, for $\tilde{s}\in \mathbb{Z}$.  
For details about notations in the next theorem, see \S\ref{sec: Peetersson-trace-formula}.
\begin{theorem}\label{Single term}
 We continue with the notations and assumptions of Theorem \ref{Main theorem}. In addition, let $F$ have odd narrow class number.  Further let $\mathfrak{b}_1\mathfrak{N}=\tilde{s}\mathcal{O} $ with $\tilde{s}\in \mathbb{Z}$. The following statements are true.
 \begin{itemize}
 \item[{\bf(a)}] The main term in Theorem \ref{Main theorem} is equal to
\begin{multline*}
\Hat{T}(m_1,m_2,\mathfrak{n})\frac{\sqrt{d_F\Nr (\mathfrak{n})}}{\omega_\mathfrak{N}(m_1/s)\omega_{\mathrm{f}}(s)} + \Bigg\{ \omega_{\mathrm{f}}(s\b_1^{-1} ) S_{\omega_\mathfrak{N}} (m_1,m_2;\eta_1 \b_1^{-2};s\b_1^{-1})  \frac{\sqrt{\Nr (\eta_1)}}{\Nr(s)} \\\times \prod_{j=1}^r\frac{2\pi}{(\sqrt{-1})^{k_j}}J_{k_j-1} \Big(  \frac{4 \pi\sqrt{\sigma_j (\eta_1 m_1m_2 )}}{|\sigma_j(s)|}\Big)
 \Bigg\} 
 \end{multline*}
 where $s=|\tilde{s}|.$
  \item[{\bf(b)}]   
Further assume that  $S_{\omega_\mathfrak{N}} (m_{1},m_{2};\eta_1 \b_1^{-2};s\b_1^{-1})\neq 0$ for all $m_1,m_2 \in \mathfrak{d}^{-1}_+$. Then, 
\begin{align*}
 \Bigg| \frac{e^{2\pi tr_{\mathbb{Q}}^F ( m_{1}+ m_{2})}}{{\psi(\mathfrak{N})}} \prod_{j=1}^r  \frac{(k_j-2)!}{(4\pi \sqrt{|\sigma_j( m_{1} m_{2})|})^{k_j-1}} 
& \sum_{{\phi}\in \mathcal{F} }\frac{\lambda_\mathfrak{n}^\phi W_{ m_{1}}^\phi(1) \overline{W_{ m_{2}}^\phi(1)}}{\|\phi\|^2}\\ - \Hat{T}( m_{1}, m_{2},\mathfrak{n})\frac{\sqrt{d_F\Nr(\mathfrak{n})}}{\omega_\mathfrak{N}( m_{1}/s)\omega_{\mathrm{f}}(s)}\Bigg| 
 &\gg \prod_{j=1}^r (k_j-1)^{-\frac{1}{3}},
\end{align*}
as $k_0\rightarrow  \infty.$
\end{itemize}
\end{theorem}
\begin{remark}
        For squarefree $N$, the main term in \cite[Theorem 1.7]{JS} is $$2\pi i^{-k}\frac{\mu(N)}{N} \prod_{p\mid N}  (1-p^{-2})J_{k-1}(4\pi\sqrt{mn}). $$ Note that $\mu(N)=S(1,0, N),$ where $S(a,b,c)$ denotes the classical Kloosterman sum, which is nonzero if and only if $N$ is squarefree.

In our context, we show  in Lemma \ref{Kloosterman sum nonzero} that for some specific values of $\eta_1,\b_1$ and $s$ and the trivial character $\omega_\mathfrak{N}$, the Kloosterman sum  $S_{\omega_\mathfrak{N}} (m_{1},m_{2};\eta_1 \b_1^{-2};s\b_1^{-1})\neq 0$ for all $m_1,m_2 \in \mathfrak{d}^{-1}_+$.  
\end{remark}
The main idea behind the proof of Theorem \ref{Single term} is to reduce the triple sum 
$$
\sum_{i=1}^t \sum_{u\in U, \eta_i u\in F^+}\sum_{s\in A_i }
$$
in Theorem \ref{Main theorem}  into a single term where finding a lower bound is possible. Similar lower bounds are obtained in \cite{Das2026} when $F$ is a multi-quadratic field.

   Now we consider the following application of Theorem \ref{Single term}. An analogue of $\nu_{k, N}$ for Hilbert cusp forms can be defined in the following manner. Let 
\begin{equation}\label{nu-HF-2}
\tilde{\nu}_{k,\mathfrak{N}}:=\prod_{j=1}^r  \frac{(k_j-2)!}{  (4\pi)^{k_j-1}}
\sum_{{\phi}\in \mathcal{F} }\frac{\delta_{\kappa^\phi_{\mathfrak{p}}} }{\|\phi\|^2}, 
\end{equation}
where  $\mathcal{F}$ is an orthogonal basis for $A_{k}(\mathfrak{N},\omega)$ consisting of eigenfunctions of the Hecke operator $T_{\mathfrak{p}}$ with $(\mathfrak{N},\mathfrak{p})=1.$
\begin{remark}
    Note that the measure considered in \eqref{nu-HF-2} is a variant of the measure considered \eqref{nu-HF-1}. To see the choices for $m_1, m_2$ and $\mathfrak{n}$ that were made to obtain $\tilde{\nu}_{k,\mathfrak{N}}$, the reader is referred to \S\ref{sec: proof-Main-theorem-2}.
\end{remark}

We obtain the following variant of \eqref{Discrepancy} (\cite[Theorem 1.6]{JS}) for the space $A_k(\mathfrak{N},1)$.

\begin{theorem}\label{Main theorem 2}
   We continue with the notations and assumptions of Theorem \ref{Main theorem}. Let $F$ have narrow class number $1$. Let  $\mathfrak{b}_1\mathfrak{N}=\tilde{s}\mathcal{O} $ with $\tilde{s}\in \mathbb{Z}$ such that $|\tilde{s}|$ is squarefree and  $\omega_{\mathfrak{N}}$ is trivial.  Let   $\mathfrak{p} = \tilde{p}\mathcal{O}$ be a prime ideal such that $|\sigma_j(\tilde{p})| > 1 $ for all $j = 1, \dots, r$.
  There exists an infinite sequence of weights $k_l=(k_{l_1},...,k_{l_r})$ with $(k_{l})_0\rightarrow \infty $  such that  $$D(\tilde{\nu}_{k_l,\mathfrak{N}},\mu_\infty)\gg\frac{1}{\big( \log (k_{l})_0\big)^2  \times\prod_{i=1}^r (k_{l_i}-1)^{\frac{1}{3}}}
 $$
 as $l \to \infty$.
\end{theorem}

\begin{remark}
    \begin{itemize}
        \item It is conjectured that there are infinitely many real quadratic fields with narrow class number $1$. Some examples of such fields are given in the OEIS sequence A003655.
        \item  The assumption that $|\sigma_j(\tilde{p})| > 1 $ for all $j = 1, \dots, r$ is satisfied for a rational prime $\tilde{p}$ inert in $F$.
        \item In the above Theorem, we only prove the discrepancy result for the weighted measure $\tilde{\nu}_{k_l,\mathfrak{N}}$. We expect a similar result for the unweighted measure as well (similar to \cite{JS}), which we hope to prove in future work.
    \end{itemize}
\end{remark}

\subsubsection*{Organisation of this article} In \S  \ref{sec: Peetersson-trace-formula}, we set up some basic notations and definitions for Hilbert modular forms. We also recall the Petersson trace formula for these forms and a related estimate from \cite{KL}. In \S \ref{sec: error-term-estimate}, we prove Theorem \ref{Main theorem}, which is an estimate of the error term in the Petersson trace formula. In \S \ref{sec: special-cases}, we analyse the main term from Theorem \ref{Main theorem} in special cases and prove Theorem \ref{Single term}. Finally, in \S \ref{sec: proof-Main-theorem-2}, we prove Theorem \ref{Main theorem 2} which gives an $\Omega$-estimate for the discrepancy between the measures $\tilde{\nu}_{k,\mathfrak{N}}$ and $\mu_\infty$.

\bigskip

\section{Petersson trace formula}\label{sec: Peetersson-trace-formula}
In this section, we recall some basic facts about Hilbert modular forms and the Petersson trace formula for the space $A_k(\mathfrak{N},\omega)$. 

 Let $F $ be a totally real number field and $r=[F :\Q].$  Let $\sigma_1, \dots,  \sigma_r$ be the distinct embeddings of $F \hookrightarrow \mathbb{R}.$ Let $\sigma: F\rightarrow \R^r$ be given by $\sigma(s)=(\sigma_1(s),\,...\,,\sigma_r(s))$.  Let $\infty_1, \dots,\infty_r$ denote the corresponding Archimedean valuations. Let $\mathcal{O}$ be the ring of integers of $F$. Let $\Nr :F\rightarrow \mathbb{Q}$ denote the norm map. For a nonzero ideal $\mathfrak{a}\subset \mathcal{O} $, let $\Nr(\mathfrak{a})=|\mathcal{O}/\mathfrak{a}|.$ For $\alpha \in F^\times,$ we have
$\Nr(\alpha\mathcal{O})=|\Nr(\alpha)|.$  Let $d_F$ denote the discriminant of $F$.

Let $\nu=\nu_{\gp}$ be the discrete valuation corresponding to a prime ideal $\gp$. Let $F_\nu$ be the completion of $F$ with respect to the valuation $\nu$. Let $\mathcal{O}_\nu $ be the ring of integers of the local field $F_\nu.$
Let $\A$ denote the ad\`ele ring of $F$ with finite ad\`eles $\A_{f}$, so that $\A=F_{\infty} \times \A_{f}$ where $F_\infty =F \otimes \mathbb{R} \cong \mathbb{R}^r$. Let $\hat{\mathcal{O}}=\prod_{\nu<\infty} \mathcal{O}_\nu \subset \A_f.$ Let $F^+$ denote the set of totally positive elements of $F$, i.e., all $x\in F$ such that $\sigma_i(x)>0 $ for all $i=1, \dots, r$.  We let $F^+_\infty$ denote the subset of $F_\infty$ consisting of vectors whose entries are all positive. Let $\mathfrak{d} ^{-1}=\{x\in F \, : \,  \mathrm{Tr}_{\mathbb{Q}}^{F}  (x\mathcal{O})\subset \mathbb{Z}  \}$
denote the inverse different ideal. We also let $\mathfrak{d}_+^{-1}=\mathfrak{d}^{-1}\cap F^+$.

Let $\mathfrak{N}$ be an integral ideal of $\mathcal{O}$. Let $k=(k_1,\,...\,,k_r) $ be an $r$-tuple of integers with $k_j\geq 2$. Let $\omega : F^\times \backslash \A^\times \rightarrow \mathbb{C}^\times$ be a unitary Hecke character. We can decompose $\omega$ as $\omega=\prod_\nu \omega_\nu$, where $\omega_\nu: F_\nu^\times \to \mathbb C^\times$ are the local characters. We further assume that
\begin{enumerate}
\item the conductor of $\omega$ divides $\mathfrak{N}$,
\item $\omega_{\infty_j}(x)=\mathrm{sgn} (x)^{k_j} $ for all $j=1,\dots, r$.
\end{enumerate}
The first condition means that $\omega_\nu $ is trivial on $1+\mathfrak{N} \mathcal{O}_\nu$ for all $\nu | \mathfrak{N},$ and unramified for all $\nu \nmid \mathfrak{N} .$

 Let $\theta : \A \rightarrow \mathbb{C}^\times$ be the standard character of $\A$. Concretely, $\theta (x) = \theta_\infty (x_\infty) \cdot \prod_{\nu < \infty} \theta_\nu (x_\nu)$, where 
 \begin{enumerate}
\item $\theta_\infty: F_\infty \to \mathbb C^\times$ is defined by $\theta_\infty (x_\infty)=e^{-2\pi i( x_1+ \cdots +x_r)}$ for $x_\infty = (x_1,...,x_r) $, and
 \item for $\nu<\infty$, $\theta_\nu : F_\nu \rightarrow \mathbb{C}^\times$ is given by  $\theta_\nu (x_\nu)=e^{2\pi i \{\mathrm{Tr}_\nu (x_\nu)\}}$.  
 Here $\{\mathrm{Tr}_\nu (x_\nu)\}$ is obtained by composing the following maps: $\mathrm{Tr}_{\Q_p}^{F_\nu} : F_\nu \to \Q_p$, going modulo $p$-adic integers: $\Q_p \to \mathbb Z_p$, and identifying $\Q_p/\mathbb Z_p$ with $\Q/\mathbb Z$. This map is well-defined since $e^{2\pi i \mathbb Z} = 1$.  Moreover, the kernel of $\theta_\nu$ is the local inverse different $\mathfrak{d}_\nu^{-1}=\{x\in F_\nu :  \mathrm{Tr}_{\mathbb{Q}_p}^{F_\nu}  (x)\in \mathbb{Z}_p\}$. 
 \end{enumerate}
Note that $\theta(x)=1$ for $x\in F$, therefore $\theta$ is a character on $F \backslash \A$.

We now define Kloosterman sums, first locally and then globally. For any finite valuation $\nu$ of $F$, let $\mathtt{n}_\nu\in \mathcal{O}_\nu,$ $\mathtt{n}_\nu\neq 0, $ and $\mathtt{m}_{1\nu},\mathtt{m}_{2\nu}\in\mathfrak{d}_\nu^{-1}$.  For $\mathtt{c}_\nu \in \mathfrak{N}_\nu$, $\mathtt{c}_\nu\neq 0$ , we define the local Kloosterman sum 
\begin{equation}\label{KS-local}
S_{\omega_\nu} (\mathtt{m}_{1\nu},\mathtt{m}_{2\nu};\mathtt{n}_\nu;\mathtt{c}_\nu) = \sum_{\substack{ s_{1},s_{2}\in \mathcal{O}_\nu / \mathtt{c}_\nu \mathcal{O}_\nu  \\{s_1s_2\equiv \mathtt{n}_\nu \ \mathrm{mod}\ c_\nu \mathcal{O}_\nu }} }\theta_\nu \Big(\frac{\mathtt{m}_{1\nu} s_{1}+\mathtt{m}_{2\nu} s_{2}}{\mathtt{c}_\nu}\Big)\omega_\nu(s_2)^{-1}. 
\end{equation}
The value of the sum is $1$ if $\mathtt{c}_\nu \in \mathcal{O}_\nu^\times$.

For a fractional ideal $\mathfrak{a},$ let $\Hat{\mathfrak{a}}=\prod_{\nu<\infty} \mathfrak{a}_\nu \subset \A_f$ where  $\mathfrak{a}_\nu$ denotes its localisation. Let 
$$
\theta_f(x)=\prod_{\nu < \infty} \theta_\nu (x_\nu),
$$ 
and let 
$$
\omega_{\mathfrak{N}}=\prod\omega_{\mathfrak{N},\nu}= \prod_{\nu|\mathfrak{N}} \omega_\nu.
$$
where $\omega_{\mathfrak{N},\nu}=\omega_\nu$ if $\nu| \mathfrak{N}$, and $1$
if $\nu \nmid \mathfrak{N}$. 
For $\mathtt{n}\in \Hat{\mathcal{O}}\cap \A^\times_{f}$, $\mathtt{c} \in \Hat{\mathfrak{N}} \cap \A^\times_{f}$, and $\mathtt{m_1},\mathtt{m_2}\in \Hat{\mathfrak{d}}^{-1},$  we define the global Kloosterman sum
 \begin{equation}\label{KS-global}
 S_{\omega_{\mathfrak{N}}} (\mathtt{m}_{1},\mathtt{m}_{2};\mathtt{n};\mathtt{c})= \sum_{\substack{ s_{1},s_{2}\in \Hat{\mathcal{O}} / \mathtt{c} \Hat{\mathcal{O}}  \\{s_1s_2\equiv \mathtt{n} \ \mathrm{mod}\ \mathtt{c}  \Hat{\mathcal{O}}}} } \theta_{f} \Big(\frac{\mathtt{m}_{1} s_{1}+\mathtt{m}_{2} s_{2}}{\mathtt{c}}\Big)\omega_{\mathfrak{N}}(s_2)^{-1},
 \end{equation}
We have the following relation between the global and local Kloosterman sums, 
$$ 
S_{\omega_{\mathfrak{N}}} (\mathtt{m}_{1},\mathtt{m}_{2};\mathtt{n};\mathtt{c})=\prod_{\nu <\infty} S_{\omega_{\mathfrak{N},\nu}} (\mathtt{m}_{1\nu},\mathtt{m}_{2\nu}; \mathtt{n}_\nu; \mathtt{c}_\nu).
$$
Note that the product on the RHS above is well-defined because $\mathtt{c}_\nu \in \mathcal{O}_\nu^\times$ except for finitely many $\nu$. 

Let $K_{f}=\prod_{\nu<\infty} \GL_2(\mathcal{O}_\nu)$ be the standard maximal compact subgroup of $\GL_2 (\A_{f}) .$
Let 
\begin{equation}\label{K_1-def}
K_1(\mathfrak{N})=  \left\{ \begin{pmatrix} a & b\\  c & d \end{pmatrix} \in K_{f}  : c \in \mathfrak{N} \Hat{\mathcal{O}}, d\in 1+ \mathfrak{N}\Hat{\mathcal{O}} \right\},
\end{equation} and let   $A_k(\mathfrak{N},\omega) $ be the  space  of Hilbert cusp forms with respect to  $K_1(\mathfrak{N})$, of weight $k$ and central character $\omega$. We also define $K_0(\mathfrak{N})=  \left\{ \begin{pmatrix} a & b\\  c & d \end{pmatrix} \in K_{f}  : c \in \mathfrak{N}\Hat{\mathcal{O}}  \right\}$ and let $$\psi(\mathfrak{N})=[K_{f} : K_0(\mathfrak{N})]=\Nr(\mathfrak{N})\prod_{\mathfrak{p}|\mathfrak{N}}\Big(1+\frac{1}{\Nr(\mathfrak{p})}\Big).$$

For $\phi \in A_k(\mathfrak{N},\omega)$ and $m\in \mathfrak{d}^{-1}_+$, we now recall the definition of the $m$th Fourier coefficient of $\phi$, denoted by  $W_m^\phi$, as in \cite[\S 3.4]{KL}. Every character on $F\backslash\A$ is of the form $\theta_m(x)=\theta(-mx)$ for some $m\in F.$
Consider the unipotent subgroup $\tilde{N}=\bigg\{\begin{pmatrix}
1 & * \\
 & 1 
\end{pmatrix}\bigg\}$   of $\GL_2$. As topological groups, $\tilde{N}(\A)\cong \A,$ so characters of the two can be identified. Hence for $m\in F,$ we identify $\theta_m$ with  a character on $\tilde{N}(F)\backslash \tilde{N}(\A)$ in the obvious way. For any $g\in\GL_2(\A)$, the map $n\mapsto \phi(ng)$ is a continuous function on $\tilde{N}(F)\backslash \tilde{N}(\A),$ with a Fourier expansion
$$
\phi\bigg(\begin{pmatrix}
1 & x \\
 & 1 
\end{pmatrix}g \bigg)=\frac{1}{\sqrt{d_F}}\sum_{m\in F} W_m^\phi(g) \theta_m(x).
$$  
The coefficients are Whittaker functions defined by 
\begin{equation}\label{eqn: Fourier-coefficient}
W_m^\phi(g)=\int_{F\backslash\A}     \phi\bigg(\begin{pmatrix}
1 & x \\
 & 1 
\end{pmatrix}g \bigg) \theta(mx) \, dx     
\end{equation}
where $dx$ denotes the Lebesgue measure. 
For $y\in  \A^\times$, we also denote 
$$
W^\phi_m(y)=W^\phi_m\bigg(\begin{pmatrix}
y &  \\
 & 1 
\end{pmatrix}\bigg),
$$ 
ignoring the abuse of notation. We identify $\A_f^\times  $ with $\{1_\infty \}\times \A_f^\times \subset \A^\times,$ so that  $y\in  \A_f^\times$ can be identified as $\{1_\infty\}\times y$ in $\A^\times$.

We refer the reader to  \cite[Section 4.3]{KL} for the definition of the Hecke operators $T_{\mathfrak{n}}$ acting on $A_k (\mathfrak{N}, \omega)$.
 We now recall the following result.

\begin{lemma}[{{\cite[Cor.\ 4.8]{KL}}}]\label{Fourier-Coefficient-formula}
Let $\tilde{d}\in \A^\times_{f}$ such that $\tilde{d} \hat{\mathcal{O}}=\hat{\mathfrak{d}}$ and $m\in \mathfrak{d}^{-1}_+$. If $(m\mathfrak{d},\mathfrak{N} )=1,$ then for any $T_{m\mathfrak{d}}$-eigenfunction $\phi\in A_k(\mathfrak{N},\omega)$ with $W_1^\phi(1/\tilde{d})=1$  and $T_{m\mathfrak{d}}\phi=\lambda^\phi_{m\mathfrak{d}} \phi,$ we have
$$ W_m^\phi(1)=\frac{e^{2\pi r} \prod_{j=1}^r   \sigma_j(m)^{(k_j/2)-1}    }{d_Fe^{2\pi \mathrm{Tr}_{\mathbb{Q}}^F (m)}}\lambda^\phi_{m\mathfrak{d}}.$$ 
 \end{lemma}
 
We now recall the statement of Petersson's trace formula in the Hilbert modular setting. 
\begin{theorem}[{{\cite[Thm.\ 5.11]{KL}}}]\label{Petersson-Trace-Formula}
Let $\mathfrak{n}$ and $\mathfrak{N}$ be integral ideals with $(\mathfrak{n},\mathfrak{N})=1.$ Let $k=(k_1,...,k_r)$ with all $k_j>2.$ Let $\mathcal{F}$ be an orthogonal basis for $A_{k}(\mathfrak{N},\omega)$ consisting of eigenfunctions for the Hecke operator $T_\mathfrak{n}.$ Then for any $m_1,m_2\in \mathfrak{d}_+^{-1}$, we have 
\begin{multline}  \label{PTF}
\frac{e^{2\pi \mathrm{Tr}_{\mathbb{Q}}^F (m_1+m_2)}}{{\psi(\mathfrak{N})}} \Bigg[\prod_{j=1}^r  \frac{(k_j-2)!}{(4\pi \sqrt{\sigma_j(m_1m_2)})^{k_j-1}}\Bigg] \sum_{{\phi}\in \mathcal{F} } \frac{\lambda^\phi_\mathfrak{n} W_{m_1}^\phi(1) \overline{W_{m_2}^\phi(1)}}{\| \phi \|^2} 
=  \Hat{T}(m_1,m_2,\mathfrak{n}) \frac{\sqrt{d_F\Nr(\mathfrak{n})}}{\omega_\mathfrak{N}(m_1/s)\omega_{f}(s)} \\
+ \sum_{i=1}^t\sum_{\substack{u\in U \\ \eta_i u\in F^+}} \sum_{\substack{s\in \mathfrak{b_i}\mathfrak{N}/\pm \\ s\neq 0 }} \Bigg\{ \omega_{f}(sb_i^{-1} ) S_{\omega_\mathfrak{N}} (m_1,m_2;\eta_i u \mathrm{b}_i^{-2};s\mathrm{b}_i^{-1})  \\
 \times  \frac{\sqrt{\Nr(\eta_i u)}}{\Nr(s)} \times\prod_{j=1}^r\frac{2\pi}{(\sqrt{-1})^{k_j}}J_{k_j-1} \left(  \frac{4 \pi\sqrt{\sigma_j (\eta_i u m_1m_2 )}}{|\sigma_j(s)|} \right) \Bigg\}.
\end{multline}

\begin{itemize}
     \item where $\Hat{T}(m_1,m_2,\mathfrak{n})\in \{0,1\}$ is non zero if and only if there exists $s\in \Hat{\mathfrak{d}}^{-1}$ such that $m_1m_2
\in  s\Hat{\mathcal{O}}   $ and $m_1m_2 \Hat{\mathcal{O}}=s^2\Hat{\mathfrak{n}}$,
        \item $U$ is a set of representative for $\mathcal{O}^\times/{\mathcal{O}^\times}^2$,
        \item $\mathrm{b_i}\Hat{\mathcal{O}}=\Hat{\mathfrak{b}}_i$ for integral ideal $\mathfrak{b}_i$  and $i=1,\,\dots\,,t  $, where $\mathfrak{b_i}$s are a set of representatives for the equation $[\mathfrak{b}]^2[\mathfrak{n}]=1$ in the ideal class group,  
        \item $\eta_i\in F$ generates the principal ideal $\mathfrak{b}_i^2\mathfrak{n}$,
        \item $\omega_\mathfrak{N}=\prod_{v|\mathfrak{N}}\omega_v \prod_{v \nmid \mathfrak{N}} 1$,
    \end{itemize}
    
   \end{theorem} 
    
    \begin{corollary}\label{sec 2 corollary}
Let $\tilde{d}\in \A^\times_{f}$ such that $\tilde{d} \hat{\mathcal{O}}=\hat{\mathfrak{d}} $ and   $m_1,m_2\in \mathfrak{d}_+^{-1}$ such that $(m_1\mathfrak{d},\mathfrak{N} )=(m_2\mathfrak{d},\mathfrak{N})=1.$ Let     $W_1^\phi(1/\tilde{d})=1$ for all $\phi\in \mathcal{F},$ where 
 $\mathcal{F}$ is an orthogonal basis for $A_{k}(\mathfrak{N},\omega)$.  
Then 

\begin{multline*} 
\frac{e^{4\pi r}}{\psi(\mathfrak{N})d_F^2\sqrt{\Nr(m_1m_2)}}\Bigg[\prod_{j=1}^r  \frac{(k_j-2)!}{(4\pi)^{k_j-1} }\Bigg]
\sum_{{\phi}\in \mathcal{F} }\frac{\lambda^\phi_{m_1\mathfrak{d}} \overline{\lambda^\phi_{m_2\mathfrak{d}} }}{\| \phi \|^2}
=  \Hat{T}(m_1,m_2,\mathcal{O})\frac{\sqrt{d_F}}{\omega_\mathfrak{N}(m_1/s)\omega_{f}(s)} \\
+ \sum_{i=1}^t\sum_{\substack{u\in U \\ \eta_i u\in F^+}} \sum_{\substack{s\in \mathfrak{b_i}\mathfrak{N}/\pm \\ s\neq 0} } \Bigg\{ \omega_{f}(s\b_i^{-1} ) S_{\omega_\mathfrak{N}} (m_1,m_2;\eta_i u \b_i^{-2};s\b_i^{-1})
 \frac{\sqrt{\Nr(\eta_i u)}}{\Nr(s)}  \\
 \times \prod_{j=1}^r\frac{2\pi}{(\sqrt{-1})^{k_j}}J_{k_j-1} \left(  \frac{4 \pi\sqrt{\sigma_j (\eta_i u m_1m_2 )}}{|\sigma_j(s)|} \right) \Bigg\}.
\end{multline*}
    \end{corollary}
    
    \begin{proof}
    This corollary is proved by substituting the expressions for $W_{m_1}^\phi(1)$ and $W_{m_2}^\phi(1)$ obtained from  Lemma \ref{Fourier-Coefficient-formula} and taking $\mathfrak{n}=\mathcal{O}$ in Theorem \ref{Petersson-Trace-Formula}. Using  Lemma \ref{Fourier-Coefficient-formula}  we have
    $$ 
    W_{m_1}^\phi(1)=\frac{e^{2\pi r} \prod_{j=1}^r   \sigma_j(m_1)^{(k_j/2)-1}    }{d_Fe^{2\pi \mathrm{Tr}_{\mathbb{Q}}^F (m_2)}}\lambda^\phi_{m_1\mathfrak{d}}
    $$  
    and 
    $$ 
    W_{m_2}^\phi(1)=\frac{e^{2\pi r} \prod_{j=1}^r   \sigma_j(m_2)^{(k_j/2)-1}    }{d_Fe^{2\pi \mathrm{Tr}_{\mathbb{Q}}^F (m_2)}}\lambda^\phi_{m_2\mathfrak{d}}.
    $$ 
    Multiplying the first equation with the conjugate of the second one, we get 
    \begin{align*}
    W_{m_1}^\phi(1) \overline{ W_{m_2}^\phi(1)} &= \frac{e^{4\pi r}  \prod_{j=1}^r   \sigma_j(m_1m_2)^{(k_j/2)-1}}{d_F^2e^{2\pi \mathrm{Tr}_{\mathbb{Q}}^F (m_1+m_2)}} \lambda^\phi_{m_1\mathfrak{d}} \overline{\lambda^\phi_{m_2\mathfrak{d}}} \\
 &=\frac{e^{4\pi r}  \prod_{j=1}^r   \sigma_j(m_1m_2)^{((k_j-1)/2)}}{d_F^2  \sqrt{\prod_{j=1}^r   \sigma_j(m_1m_2)} e^{2\pi \mathrm{Tr}_{\mathbb{Q}}^F (m_1+m_2)}}\lambda^\phi_{m_1\mathfrak{d}} \overline{\lambda^\phi_{m_2\mathfrak{d}}}.
    \end{align*}
    Thus 
    $$
    \frac{e^{2\pi \mathrm{Tr}_{\mathbb{Q}}^F (m_1+m_2)}    }{\prod_{j=1}^r   \sigma_j(m_1m_2)^{((k_j-1)/2)}}  W_{m_1}^\phi(1)\overline{ W_{m_2}^\phi(1)} =\frac{  e^{4\pi r}  }{d_F^2 \sqrt{\Nr(m_1m_2)}} \lambda^\phi_{m_1\mathfrak{d}} \overline{\lambda^\phi_{m_2\mathfrak{d}} }.
    $$
 In Theorem \ref{Petersson-Trace-Formula}, taking $\mathfrak{n}=\mathcal{O}$ implies $\lambda^\phi_{\mathcal{O}}=1 $ 
    and the proof is complete. 
\end{proof}
    
\begin{corollary}
Let $\mathfrak{p}$ be a prime ideal not dividing the level $\mathfrak{N}.$   Let $k=(k_1,...,k_r)$ with all $k_j>2$. We also take an integer $\ell \geq 0.$ Let $\mathcal{F}$ be an orthogonal basis for $A_{k}(\mathfrak{N},\omega)$ consisting of eigenfunctions for the Hecke operator $T_\mathfrak{\mathfrak{p}^\ell}.$ Then for any $m\in \mathfrak{d}_+^{-1}$, we have
\begin{multline*} 
\frac{e^{4\pi \mathrm{Tr}_{\mathbb{Q}}^F (m)}}{\psi(\mathfrak{N})}\Bigg[\prod_{j=1}^r  \frac{(k_j-2)!}{(4\pi \sigma_j(m))^{k_j-1}}\Bigg]
\sum_{{\phi}\in \mathcal{F} }\frac{\lambda^\phi_\mathfrak{\mathfrak{p}^\ell}  |W_{m}^\phi(1)|^2 }{\|\phi \|^2}
=  \Hat{T}(m,m,\mathfrak{\mathfrak{p}^\ell}) \frac{\sqrt{d_F\Nr(\mathfrak{p}^\ell)}}{\omega_\mathfrak{N}(m/s)\omega_{f}(s)} \\
+\sum_{i=1}^t\sum_{\substack{u\in U \\ \eta_i u\in F^+}} \sum_{\substack{s\in \mathfrak{b_i}\mathfrak{N}/\pm \\ s\neq 0 }}\Bigg\{\omega_{f}(s\b_i^{-1} ) S_{\omega_\mathfrak{N}} (m,m;\eta_i u \b_i^{-2};s\b_i^{-1})
\frac{\sqrt{\Nr(\eta_i u)}} {\Nr(s)} \\
 \times \prod_{j=1}^r\frac{2\pi}{(\sqrt{-1})^{k_j}}J_{k_j-1} \left(  \frac{4 \pi    |\sigma_j(m)| \sqrt{\sigma_j (\eta_i u  )}}{|\sigma_j(s)|} \right) \Bigg\}.
\end{multline*}
\end{corollary}    
    
    \begin{proof}
    We obtain this corollary by taking $\mathfrak{n}=\mathfrak{p}^\ell$ and $m_1=m_2=m $   in Theorem \ref{Petersson-Trace-Formula}. 
    \end{proof}
\bigskip

\section{Estimating the error term in the trace formula}\label{sec: error-term-estimate}

We prove Theorem \ref{Main theorem} in this section.  First, we estimate the triple sum appearing on the right-hand side of the trace formula given by Theorem \ref{Petersson-Trace-Formula}. One of the main objects appearing in the sum is the $J$-Bessel function. Hence,
in the next lemma, we recall bounds for the $J$-Bessel function of the first kind that will be useful. 
\begin{lemma}\label{Bessel-estimate}
We have the following estimates of the $J$-Bessel function. 
\begin{enumerate}[label=(\roman*)]
    \item \label{Bessel-estimate-part-1} If $a \geq 0$ and $0<x\leq 1$, we have $$1\leq \frac{J_a(ax)}{x^aJ_a(a)}\leq e^{a(1-x)}.$$
    \item \label{Bessel-estimate-part-2}   As $a\rightarrow \infty $, we have $0<J_a(a) \ll \frac{1}{a^{\frac{1}{3}}}$. 
    \item \label{Bessel-estimate-part-3}  If $|d|<1$, then $$\frac{1}{a^{\frac{1}{3}}} \ll  J_a(a+da^{\frac{1}{3}})\ll \frac{1}{a^{\frac{1}{3}}}. $$
    \item \label{Bessel-estimate-part-4}  For $x\in \mathbb{R}$ and $a > 0$, $|J_a(x)|\leq \min(ba^{-\frac{1}{3}},c|x|^{-\frac{1}{3}})$ where $b \approx 0.674885$ and $c \approx 0.7857468704$. 
    \item \label{Bessel-estimate-part-5} For $\frac{1}{2} \leq x<1 $, we have the following uniform bound $$J_a(ax)\ll \frac{1}{(1-x^2)^{1/4}a^{1/2}}. $$
\end{enumerate}
\end{lemma}

\begin{proof}
We refer to \cite[\S 2.1.1]{JS}  for all parts of the lemma, except (iv), for which we refer to \cite[\S 1]{LJ}. Part (iii) is an essential ingredient for obtaining the lower bound in Theorem 2. The geometric origin of the transition behavior of the $J$-Bessel function given by (iii) is explained by Marshall in \cite[\S 5]{JS}.
\end{proof}

We now give an outline of the proof of Theorem \ref{Main theorem}.

\subsection*{Outline of the proof of Theorem \ref{Main theorem}.}
In  \S\ref{sec 2}, we first fix $i$ in equation \eqref{PTF}, where $i$ varies over a finite set. Then we take $u\in U$ satisfying $\eta_i u\in F^+$ to be fixed till \S\ref{last step}.  This helps us to focus on a single lattice at a time and estimate the right-hand side of equation \eqref{PTF}. In \S \ref{sec 3}, we fix an orthant of the lattice so that we can apply J-Bessel function estimates carefully. This is necessary as we do not use a uniform bound for a fixed lattice.  According to the usability of Lemma \ref{Bessel-estimate}, we apply bounds for a fixed orthant in two steps.  
In \S \ref{last step}, firstly we vary over all orthants (earlier fixed in \S \ref{sec 3})  for a fixed lattice, then we vary over all possible lattices along with  $u\in U$ satisfying $\eta_i u\in F^+$ to complete the proof. 
\begin{proof}[Proof of Theorem \ref{Main theorem}]
Let $m_1$ and $m_2$ satisfy equation \eqref{m1jm2j}. 
We fix some notations before we proceed further in the proof.  
Let $\delta_i$ and $A_i$ be as in equations \eqref{delta_i} and \eqref{A_i}. 
 Recall that  $$\gamma_{j}=\max \left\{\sqrt{\sigma_{j}(\eta_i u)} \, : \,i =1,...\,,t  ,u\in U, \eta_i u\in F^+\right\} .$$ 
Further, let  $$\beta_{j}=\min \left\{\sqrt{\sigma_{j}(\eta_i u)} \, : \,i =1,...\,,t  ,u\in U, \eta_i u\in F^+\right\} \qquad \textrm{and} \qquad \epsilon_j=\frac{\gamma_j}{\beta_j} $$ 
for all $j=1,\dots, r$.

\subsection{Fixing a lattice.}\label{sec 2}
Let us fix $i$ between $1$ and $t$ in the triple sum given in equation \eqref{PTF}.  
Note that $\eta_i$ is now fixed, as $i$ is fixed. We now fix $u\in U$ satisfying $\eta_i u\in F^+$ until   \S \ref{last step}. 
We break the last of the triple sum into two parts, the first part is a summation over $s\in A_i$ and 
 and the second over $s\not\in A_i$.
By taking  $A_i^\prime=\{s\in  \mathfrak{b_i}\mathfrak{N}/\pm\, : \, s\neq 0, s\notin A_i\}$, we break down the summation 
   $$ 
 \sum_{\stackrel{s\in  \mathfrak{b_i}\mathfrak{N}/\pm,}{s\neq 0} }\omega_{\mathrm{f}}(s\b_i^{-1} ) S_{\omega_\mathfrak{N}} (m_{1},m_{2};\eta_i u \b_i^{-2};s\b_i^{-1})
 \frac{\sqrt{\Nr(\eta_i u)}}{\Nr(s)}
 \prod_{j=1}^r\frac{2\pi}{(\sqrt{-1})^{k_j}}J_{k_j-1} \left(  \frac{4 \pi\sqrt{\sigma_j (\eta_i u m_{1}m_{2} )}}{|\sigma_j(s)|}\right) 
    $$  
into summations over 
    $$  \sum_{s\in A_i } \quad \textrm{and} \quad  \sum_{s\in A_i^\prime }.$$
As $A_i$ is finite, the sum $\sum_{s\in A_i }$
is finite and is going to constitute the main term. 
So we focus on estimating the sum when the sum is taken over $s\in A_i'.$
 We show that as $k_0\rightarrow \infty,$ the sum   
\begin{equation}\label{Aiprime}
\sum_{s\in A_i' }\omega_{\mathrm{f}}(s\b_i^{-1} ) S_{\omega_\mathfrak{N}} (\mc;\eta_i u \b_i^{-2};s\b_i^{-1})
 \frac{\sqrt{\Nr(\eta_i u)}}{\Nr(s)}
 \prod_{j=1}^r\frac{2\pi}{(\sqrt{-1})^{k_j}}J_{k_j-1} \Big(  \frac{4 \pi\sqrt{\sigma_j (\eta_i u \mp )}}{|\sigma_j(s)|}\Big)
\end{equation}
$$ =\O\left(\sum_{j=1}^r \left( (k_{j}-1)^{-\frac{1}{2}}\prod_{j'=1,j'\neq j}(k_{j'}-1)^{-\frac{1}{3}}\right)\right).$$
   The above asymptotic is achieved in  equation \eqref{sinAip}.

 Using Lemma 6.1 from \cite{KL}, we get $$ \Big|S_{\omega_\mathfrak{N}} (\mc;\eta_i u \b_i^{-2};s\b_i^{-1})\Big|\leq \Nr(\eta_i u\b_i^{-2})\Nr(s\b_i^{-1}),$$  which implies   
 $$ \Bigg|\omega_{\mathrm{f}}(s\b_i^{-1} ) S_{\omega_\mathfrak{N}} (\mc;\eta_i u \b_i^{-2};s\b_i^{-1})
 \frac{\sqrt{\Nr(\eta_i u)}}{\Nr(s)}\Bigg|\leq \Nr(\eta_i u)^{\frac{3}{2}}\Nr(\b_i^{-3}),$$ and consequently
 $$
\Bigg| \sum_{s\in A_i' }\omega_{\mathrm{f}}(s\b_i^{-1} ) S_{\omega_\mathfrak{N}} (\mc;\eta_i u \b_i^{-2};s\b_i^{-1})
 \frac{\sqrt{\Nr(\eta_i u)}}{\Nr(s)}
 \prod_{j=1}^r\frac{2\pi}{(\sqrt{-1})^{k_j}}J_{k_j-1} \Big(  \frac{4 \pi\sqrt{\sigma_j (\eta_i u \mp )}}{|\sigma_j(s)|}\Big)     \Bigg|$$
\begin{align}\label{Ai'raw}
\leq 
 \sum_{s\in A_i' }  \Nr(\eta_i u)^{\frac{3}{2}}\Nr(\b_i^{-3}) \prod_{j=1}^r2\pi \Bigg| J_{k_j-1} \Big(  \frac{4 \pi \sqrt{\sigma_j (\eta_i u \mp )}}{|\sigma_j(s)|}\Big)\Bigg|
 .\end{align}
Therefore, our goal now is to find bounds for 
\begin{align*}
\sum_{s\in A^\prime_{i} }    \prod_{j=1}^r\left| J_{k_j-1} \Big(  \frac{4 \pi\sqrt{\sigma_j (\eta_i u m_1m_2 )}}{|\sigma_j(s)|}\Big)\right|.
\end{align*}

Since $s\in A_i'$, there exists $j_0$ ($j_0$ depends upon $s$) such that $|\sigma_{j_0}(s)|> 2 \tilde{\delta}_i.$ Due to the hypothesis  $$\frac{2\pi \gamma_{j_0}\sqrt{\sigma_{j_0}(m_1m_2)}}{\tilde{\delta}_i}\in \Big((k_{j_0}  -1)-(k_{j_0}  -1)^{\frac{1}{3}},(k_{j_0}  -1)\Big),$$
 we have the inequality
\begin{equation}\label{8/9}
 \frac{8}{9}< \Bigg(1-\frac{1}{(k_{j_0}-1)^{\frac{2}{3}}}\Bigg) <\Bigg|\frac{2\pi \gamma_{j_0}\sqrt{\sigma_{j_0}(m_1m_2)}}{(k_{j_0}-1)\tilde{\delta}_i}\Bigg|<1,
\end{equation}
whenever $k_{j_0}>28$.
Applying part~\ref{Bessel-estimate-part-1} of Lemma~\ref{Bessel-estimate},  we have
$$\Bigg| J_{k_{j_0}-1}\Big(\frac{4\pi\sqrt{\sigma_{j_0}(\eta_i u m_1m_2)}}{|\sigma_{j_0} (s )| }\Big)\Bigg|=\Bigg| J_{k_{j_0}-1}\Bigg((k_{j_0}-1)\frac{4\pi\sqrt{\sigma_{j_0}(\eta_i u m_1m_2)}}{(k_{j_0}-1)|\sigma_{j_0} (s )|}\Bigg)\Bigg| \leq e^{a(1-x)}x^aJ_a(a),
$$
by taking $a=k_{j_0}-1$ and $x=\frac{4\pi\sqrt{\sigma_{j_0}(\eta_i u m_1m_2)}}{(k_{j_0}-1)|\sigma_{j_0} (s )|}$. The lemma is applicable since $$x< \frac{2\pi\gamma_{j_0} \sqrt{\sigma_{j_0}( m_1m_2)}}{(k_{j_0}-1)\tilde{\delta}_i} <1$$ by equation \eqref{8/9}. 
Using part~\ref{Bessel-estimate-part-2} of Lemma \ref{Bessel-estimate}, we obtain 
\begin{align}\label{Ja(a)}
J_a(a)\ll \frac{1}{a^{\frac{1}{3}}}=\frac{1}{(k_{j_0}-1)^{\frac{1}{3}}}.\end{align}
Moreover, equation \eqref{8/9}  implies that $\frac{16\tilde{\delta}_i}{(9\epsilon_{j_0})|\sigma_{j_0} (s )|}<x<\frac{2\tilde{\delta}_i}{|\sigma_{j_0} (s )|}.$ 
Therefore, we obtain
\begin{equation}\label{Exponential bound}
e^{a(1-x)}x^a = e^{a(1-x + \log x)} < e^{a\Bigg(1-\frac{16\tilde{\delta}_i}{(9\epsilon_{j_0})|\sigma_{j_0}(s )|}+\log\Big(\frac{2\tilde{\delta}_i}{|\sigma_{j_0} (s )|}\Big)\Bigg)}.
\end{equation}

We now need a device to keep track of the place where $|\sigma_{j}(s)|> 2 \tilde{\delta}_i$. The following consideration helps us to achieve this. Let  $M>1$  be an absolute constant, suitably chosen later.  
For each $j$, the quantity $\sigma_j(s) $ falls into one of the following three categories: 
\begin{itemize}
     \item $ |\sigma_j(s)|\leq 2 \tilde{\delta}_i$,
    \item $ 2 \tilde{\delta}_i < |\sigma_{j}(s)|\leq (M+1) \tilde{\delta}_i,$ or
        \item $|\sigma_j(s)|> (M+1)\tilde{\delta}_i$.
\end{itemize}
   We then consider an $r$-tuple $(a_1,\,\dots \, , a_r) $ with each $a_j\in \{0,1,2\}$ defined as follows: 
\begin{equation}\label{aj}
 a_j=\begin{cases}
			0, & \text{if} \,\,\, |\sigma_j(s)|\leq 2 \tilde{\delta}_i,  \\
            1, & \text{if}   \,\, \, 2 \tilde{\delta}_i < |\sigma_{j}(s)|\leq (M+1) \tilde{\delta}_i, \\
            2, &  \text{if}\,\,\,  |\sigma_j(s)|> (M+1)\tilde{\delta}_i.
		 \end{cases}
\end{equation}
For each $\ell \in \{0, 1, 2\}$, let $\mathfrak{I}_\ell = \{ j \mid a_j = \ell \}$. Clearly, $\{1, \dots, r\} = \mathfrak{I}_0 \sqcup \mathfrak{I}_1 \sqcup \mathfrak{I}_2 $.

For each $r$-tuple $(a_1,...\, ,a_r)$ as above, we set 
\begin{equation}\label{h}
 h=\sum_{j=1}^r a_j 3^{j-1}. 
 \end{equation}
For each $s \neq 0$, we have the corresponding $a_j$ depending on the value of $|\sigma_j (s)|$. Therefore, we have a corresponding $h(s)$. Note that $s\in A_i$ if and only if $h(s)=0$. We now partition the set $A_i'$ according to the values of $h(s)$. Let $A'_{i,h}$ denote the subset of $A'_i$ consisting of all $s$ such that $h(s)=h$. Finally, we have 
$$
A_i'=\cup_{h=1}^{3^r-1}A'_{i,h}.
$$ 
Hence, we have 
\begin{equation}\label{eqn:Ai-prime-Ai-prime-h}
   \sum_{s \in A^\prime_i} = \sum_{h=1}^{3^r-1} \sum_{s \in A^\prime_{i,h}}. 
\end{equation}
  For a fixed $h$ (recall that $i$ is already fixed), we now find bounds for the expression   
\begin{align}\label{A'ih}
\sum_{s\in A'_{i,h} }    \prod_{j=1}^r\left| J_{k_j-1} \Big(  \frac{4 \pi\sqrt{\sigma_j (\eta_i u m_1m_2 )}}{|\sigma_j(s)|}\Big)\right|.
\end{align}

\subsubsection{Bounds for $\mathfrak{I}_0$ and $\mathfrak{I}_2$.}\label{subsec:bounds-I0-I2}

For $j \in \mathfrak{I}_0$, we use the uniform bound given by Lemma \ref{Bessel-estimate}\ref{Bessel-estimate-part-4}.
We have  $$\left| J_{k_{j}-1} \Big(  \frac{4 \pi\sqrt{\sigma_{j} (\eta_i u m_1m_2 )}}{|\sigma_{j}(s)|}\Big)\right| \leq \min \Bigg((k_{j}-1)^{-\frac{1}{3}},\Big(  \frac{4 \pi\sqrt{\sigma_{j} (\eta_i u m_1m_2 )}}{|\sigma_{j}(s)|}\Big)^{-\frac{1}{3}}\Bigg)$$
$$\ll  \min \Big((k_{j}-1)^{-\frac{1}{3}},\Big(  \sqrt{\sigma_{j} (\eta_i u m_1m_2 )}\Big)^{-\frac{1}{3}}\Big)$$
as $|\sigma_{j}(s)| \leq 2 \tilde{\delta}_i.$
By equation \eqref{8/9},  $\sqrt{\sigma_{j} (\eta_i u m_1m_2 )}\geq \frac{4\sqrt{\sigma_{j} (\eta_i u )} (k_{j}-1)\tilde{\delta}_i}{9\pi\gamma_{j}},$ which after using in the above upper bound yields  
\begin{equation}\label{bound for l}
  \left| J_{k_{j}-1} \Big(  \frac{4 \pi\sqrt{\sigma_{j} (\eta_i u m_1m_2 )}}{|\sigma_{j}(s)|}\Big)\right| \ll_{F,\mathfrak{N},\mathfrak{n}} \min\left(  (k_{j}-1)^{-\frac{1}{3}}, (k_{j}-1)^{-\frac{1}{3}}  \right ) =(k_{j}-1)^{-\frac{1}{3}}.
\end{equation}
Now for $j \in \mathfrak{I}_2$  equation \eqref{Ja(a)} and \eqref{Exponential bound} implies 
\begin{equation}\label{bound for g}
\left| J_{k_{j}-1} \Big(  \frac{4 \pi\sqrt{\sigma_{j} (\eta_i u m_1m_2 )}}{|\sigma_{j}(s)|}\Big)\right|
\leq (k_{j}-1)^{-\frac{1}{3}}e^{(k_{j}-1)\Bigg(1-\frac{16 \tilde{\delta}_i}{(9\epsilon_{j})|\sigma_{j}(s )|}+\log\Big(\frac{2\tilde{\delta}_i}{|\sigma_{j} (s )|}\Big)\Bigg)}.
\end{equation}

\subsection{Estimates of the $J$-Bessel functions}\label{sec 3}

 We consider the sum  given in  equation~\eqref{A'ih},   
 $$
 \sum_{s\in A'_{i,h} }    \prod_{j=1}^r \left| J_{k_j-1} \Big(  \frac{4 \pi\sqrt{\sigma_j (\eta_i u m_1m_2 )}}{|\sigma_j(s)|}\Big)\right|.
 $$

We can express this as 
$$
\sum_{s\in A'_{i,h} }    \prod_{j\in \mathfrak{I}_0} \prod_{j\in \mathfrak{I}_1} \prod_{j\in \mathfrak{I}_2} 
$$
by breaking the inner product into products over the three sets $\mathfrak{I}_0$, $\mathfrak{I}_1$ and $\mathfrak{I}_2$.
It is important to note that the partition of $\{1, \dots, r\}$ into $\mathfrak{I}_0$, $\mathfrak{I}_1$ and $\mathfrak{I}_2$ does not depend on $s$, as $s$ varies in the lattice $A'_{i,h}$. This is essential to apply Lemma \ref{Bessel-estimate} smoothly in a phased manner. 

We now apply estimates for the $J$-Bessel function. We break this into two cases. In the first case, the set $\mathfrak{I}_1 \not = \emptyset$ and we address this case in \S \ref{step a} below. In the second case, the set $\mathfrak{I}_1 = \emptyset$, and this case is addressed in \S \ref{step b} below.

\subsubsection{Case: $\mathfrak{I}_1 \not = \emptyset$}\label{step a}
For any $j \in \mathfrak{I}_1$, let 
 \begin{align}\label{xalpha}
\tilde{x}_{j} (s) = \frac{4 \pi\sqrt{\sigma_{j} (\eta_i u m_1m_2 )}}{(k_{j}-1) |\sigma_{j}(s)|}.
\end{align}
Note that $0 < \tilde{x}_{j} (s) < 1$. We will be applying different bounds depending on whether $0 < \tilde{x}_{j} (s) < 1/3$, $1/3 \le  \tilde{x}_{j} (s) < 1/2$ or $1/2 \le  \tilde{x}_{j} (s) < 1$.

We temporarily drop $s$ from the notation and use $\tilde{x}_{j}$ for simplicity.  When $ 1/2 \leq \tilde{x}_{j}  <1,$ we use the uniform bound given by Lemma \ref{Bessel-estimate}, part \ref{Bessel-estimate-part-5}. This gives 
\begin{align}\label{x<1}
J_{k_{j}-1}((k_j-1) \tilde{x}_{j})\ll \frac{1}{(1-\tilde{x}_{j}^2)^{\frac{1}{4}     }  (k_{j}-1)^{\frac{1}{2}}}	\ll_{\mathfrak{n}, \mathfrak{N}, F} (k_{j}-1)^{-\frac{1}{2}}. 
\end{align}
The last step is justified due to the fact $\tilde{x}_{j} (s)$ cannot be arbitrarily close to $1$ as $s$ varies over a discrete subset of $\R^r$ and  $2\tilde{\delta}_i < |\sigma_{j}(s)| \leq (M+1)\tilde{\delta}_i.$

When $1/3 \leq  \tilde{x}_{j} < 1/2,$ we use Lemma \ref{Bessel-estimate}, parts \ref{Bessel-estimate-part-1} and \ref{Bessel-estimate-part-2} to obtain
  $$ 
  J_{k_{j}-1}((k_{j}-1) \tilde{x}_{j}) \ll_{\mathfrak{n}, \mathfrak{N}, F} e^{(k_{j}-1)(1-\tilde{x}_{j} + \log \tilde{x}_{j})}\cdot \frac{1}{(k_{j}-1)^{\frac{1}{3}}} $$
\begin{align}\label{x<1/2}
\ll_{\mathfrak{n}, \mathfrak{N}, F}  e^{(k_{j}-1)(1-\frac{1}{3} + \log \frac{1}{2})}\cdot \frac{1}{(k_{j}-1)^{\frac{1}{3}}}.
\end{align}
Finally, for $0 <  \tilde{x}_{j} <1/3,$ we obtain
\begin{align}\label{x<1/3}
J_{k_{j}-1}((k_{j}-1) \tilde{x}_{j})\ll_{\mathfrak{n}, \mathfrak{N}, F} e^{(k_{j}-1)(1 + \log \frac{1}{3})}\cdot \frac{1}{(k_{j}-1)^{\frac{1}{3}}}.
\end{align}
By considering  equation \eqref{x<1}, \eqref{x<1/2} and \eqref{x<1/3}, for all indices $j \in \mathfrak{I}_1$, we obtain   
\begin{equation}\label{x-bound}
 J_{k_{j}-1} \left(  \frac{4 \pi\sqrt{\sigma_{j} (\eta_i u m_1m_2 )}}{|\sigma_{j}(s)|}\right)
\ll_{\mathfrak{n}, \mathfrak{N}, F} (k_{j}-1)^{-\frac{1}{2}}.    
\end{equation}

 Let $\tau_{i,1},...,\tau_{i,r}$ denote the lengths of the sides of the fundamental parallelepiped of the lattice $\sigma(b_i\mathfrak{N})$. Let $\tilde\tau_i=\min (\tau_{i,1},\,...\,,\tau_{i,r}) $ and $\tilde{\epsilon}_i=\Big(\frac{\tilde \tau_i}{2}\Big)^r.$ By the choice of $\tilde{\epsilon}_i$, we note that a cube of volume $\tilde{\epsilon}_i$ can contain at most one lattice point of $\sigma( \mathfrak{b_i} \mathfrak{N})$, and hence also at most one lattice point of $\sigma (A^\prime_{i,h})$.  
Using the bounds in \ref{x-bound}, we obtain

Note that since the partition $\{1, \dots, r\} = \mathfrak{I}_0 \cup \mathfrak{I}_1 \cup \mathfrak{I}_2 $ is the same for any $s \in A^\prime_{i,h}$, we can interchange the summation and the products to get
\begin{align*}
\sum_{s\in A'_{i,h} }    \prod_{j\in \mathfrak{I}_0} \prod_{j\in \mathfrak{I}_1} \prod_{j\in \mathfrak{I}_2} \left| J_{k_j-1} \Big(  \frac{4 \pi\sqrt{\sigma_j (\eta_i u m_1m_2 )}}{|\sigma_j(s)|}\Big)\right| =  \prod_{j\in \mathfrak{I}_0} \prod_{j\in \mathfrak{I}_1} \prod_{j\in \mathfrak{I}_2} \sum_{s\in A'_{i,h} }  \left| J_{k_j-1} \Big(  \frac{4 \pi\sqrt{\sigma_j (\eta_i u m_1m_2 )}}{|\sigma_j(s)|}\Big)\right|
\end{align*}

Now, if $j \in \mathfrak{I}_0$, then $|\sigma_j(s)| \in  (0, 2 \tilde{\delta}_i]$. We write this interval as a union of smaller intervals of length $\tilde \tau_i/2$. That is, 
$$
(0, 2 \tilde{\delta}_i] = \cup_{\tilde m_j} \left( \frac{(\tilde m_j - 1) \tilde \tau_i}{2}, \frac{\tilde m_j \tilde \tau_i}{2} \right]. 
$$ 
Note that there are only finitely many possible $\tilde m_j$ as the original interval $(0, 2 \tilde{\delta}_i]$ is of finite length. Similarly, if $j \in \mathfrak{I}_1$ or $j \in \mathfrak{I}_2$, then $|\sigma_j(s)| \in  ( 2 \tilde{\delta}_i, (M+1) \tilde{\delta}_i]$ or $|\sigma_j(s)| \in  ( (M+1) \tilde{\delta}_i, \infty)$, respectively. In each of these cases, we again break up the intervals into smaller ones of length $\tilde \tau_i/2$. That is, we write  
$$
( 2 \tilde{\delta}_i, (M+1) \tilde{\delta}_i] = \cup_{\tilde m_j} \left( 2 \tilde{\delta}_i + \frac{(\tilde m_j - 1) \tilde \tau_i}{2}, 2 \tilde{\delta}_i + \frac{\tilde m_j \tilde \tau_i}{2} \right]
$$
and
$$
( (M+1) \tilde{\delta}_i, \infty) = \cup_{\tilde m_j} \left( (M+1) \tilde{\delta}_i + \frac{(\tilde m_j - 1) \tilde \tau_i}{2}, (M+1) \tilde{\delta}_i + \frac{\tilde m_j \tilde \tau_i}{2} \right].
$$
The number of $\tilde{m}_j$ in the first two cases is finite and infinite in the last case. 

Now, we take $\tilde m = (\tilde m_j)_{j=1}^r$ and consider the corresponding interval $I_{\tilde m_j} =  \left( * + \frac{(\tilde m_j - 1) \tilde \tau_i}{2}, * + \frac{\tilde m_j \tilde \tau_i}{2} \right]$ as above. Then, $I_{\tilde m} = \prod_{j=1}^r I_{\tilde m_j} \subset \R^r$ is hypercube of volume $\left( \frac{\tau_i}{2} \right)^r = \tilde \epsilon_i $ and hence contains at most one lattice point of $\sigma (A^\prime_{i,h})$.

\begin{align}\label{case3a}
\sum_{s\in A'_{i,h} }    \prod_{j\in \mathfrak{I}_0} \prod_{j\in \mathfrak{I}_1} \prod_{j\in \mathfrak{I}_2} \left| J_{k_j-1} \Big(  \frac{4 \pi\sqrt{\sigma_j (\eta_i u m_1m_2 )}}{|\sigma_j(s)|}\Big)\right| \ll \prod_{j \in \mathfrak{I}_0} \prod_{j \in \mathfrak{I}_1} \prod_{j \in \mathfrak{I}_2} \sum_{\substack{\tilde m, \\ \sigma(s) \in I_{\tilde m}}} \left| J_{k_j-1} \Big(  \frac{4 \pi\sqrt{\sigma_j (\eta_i u m_1m_2 )}}{|\sigma_j(s)|}\Big)\right|
\end{align}

By using equations \eqref{bound for l} for $j \in \mathfrak{I}_0$, \eqref{bound for g} for $ j \in \mathfrak{I}_2$ and \eqref{x-bound} for $j \in \mathfrak{I}_1$ the RHS of  \eqref{case3a} is bounded by
\begin{equation}
\ll \prod_{j \in \mathfrak{I}_0} (k_j - 1)^{-\frac{1}{3}} \prod_{j \in \mathfrak{I}_1} (k_j - 1)^{-\frac{1}{2}} \prod_{j \in \mathfrak{I}_2} (k_j - 1)^{-\frac{1}{3}} \sum_{\substack{\tilde m_j, \\ \sigma_j (s) \in I_{\tilde m_j}}}  e^{(k_{j}-1)\Bigg(1-\frac{16\delta_i}{9\epsilon_{j} |\sigma_{j}(s )|}+\log\Big(\frac{2\delta_i}{|\sigma_{j} (s )|}\Big)\Bigg)}
\end{equation}

Taking $y_{j}= y_j (s) = \frac{|\sigma_{j} (s)|}{\delta_j}$, the sum above becomes 
\begin{equation}
= \prod_{j \in \mathfrak{I}_0} (k_j - 1)^{-\frac{1}{3}} \prod_{j \in \mathfrak{I}_1} (k_j - 1)^{-\frac{1}{2}} \prod_{j \in \mathfrak{I}_2} (k_j - 1)^{-\frac{1}{3}} \sum_{\substack{\tilde m_j, \\ \sigma_j (s) \in I_{\tilde m_j}}} e^{(k_{j}-1)\Bigg(1-\frac{16}{9\epsilon_{j} y_j}+\log\Big(\frac{2}{y_j}\Big)\Bigg)}
\end{equation}

Note that each of the intervals $I_{\tilde m_j}$ has length $|I_{\tilde m_j}| = \epsilon^{1/r}/\delta_i$. Therefore, the last equation becomes
\begin{equation}
= \frac{\delta_i^r}{\epsilon_i} \prod_{j \in \mathfrak{I}_0} (k_j - 1)^{-\frac{1}{3}} \prod_{j \in \mathfrak{I}_1} (k_j - 1)^{-\frac{1}{2}} \prod_{j \in \mathfrak{I}_2} (k_j - 1)^{-\frac{1}{3}} \sum_{\substack{\tilde m_j, \\ \sigma_j (s) \in I_{\tilde m_j}}}  e^{(k_{j}-1)\Bigg(1-\frac{16}{9\epsilon_{j} y_j}+\log\Big(\frac{2}{y_j}\Big)\Bigg)} |I_{\tilde m_j}|
\end{equation}

Note that for $j \in \mathfrak{I}_0,\mathfrak{I}_1$ and $\mathfrak{I}_2$, the variables $y_j$ belong to the intervals $(0, 2], (2, M+1]$ and $(M+1, \infty)$ respectively. We denote these intervals by $I_j$. For $j \in \mathfrak{I}_0$ and $\mathfrak{I}_1$, we also denote by $\tilde I_j$ the slightly larger intervals $(a_j, b_j + \frac{\epsilon_i^{\frac{1}{r}}}{\delta_i}]$, where $I_j = (a_j, b_j]$. For $j \in \mathfrak{I}_2$, we set $\tilde I_j = I_j$. Since 
$$
e^{(k_{j}-1) \Big (-\frac{16}{ 9\epsilon_{j} y_{j}}+\log\Big(\frac{2}{y_{j}}\Big)\Big)}
$$ 
is a monotonically decreasing function of $y_j$ for $y_{j}>2,$
using the integral test, the above sum is bounded by the multiple integral

\begin{align*}
 \leq \frac{\delta_i^r}{\epsilon_i} & \left( \prod_{j \in \mathfrak{I}_0} (k_j - 1)^{-\frac{1}{3}}  \prod_{j \in \mathfrak{I}_1} (k_j - 1)^{-\frac{1}{2}} \idotsint\limits_{\tilde I_j} \prod dy_j\right) \\  &\left(\prod_{j \in \mathfrak{I}_2} (k_j - 1)^{-\frac{1}{3}} \idotsint\limits_{\tilde I_j}  e^{(k_{j}-1)\Bigg(1-\frac{16}{9\epsilon_{j} y_j}+\log\Big(\frac{2}{y_j}\Big)\Bigg)} \prod dy_j \right). 
\end{align*}
By absorbing some absolute constants, the above is
\begin{align}\label{alp}
 \ll   \left( \prod_{j \in \mathfrak{I}_0} (k_j - 1)^{-\frac{1}{3}}  \prod_{j \in \mathfrak{I}_1} (k_j - 1)^{-\frac{1}{2}}  \right) \left(\prod_{j \in \mathfrak{I}_2} (k_j - 1)^{-\frac{1}{3}} \int_M^\infty e^{(k_{j}-1)\Bigg(1-\frac{16}{9\epsilon_{j} y_j}+\log\Big(\frac{2}{y_j}\Big)\Bigg)} dy_j \right). 
\end{align}

Now, for a fixed $j \in \mathfrak{I}_2$, consider the integral 
$$
\int_M^\infty e^{(k_j - 1)\ \left( -\frac{16}{ 9\epsilon_{j} y_{j}}+ \log \left(\frac{2}{y_{j}}\right) \right)} dy_{j}.
$$ 

By  substituting $x_j = \frac{16}{9\epsilon_{j} y_{j}}$ and $dx_{j} = -\frac{16 dy_{j}}{9\epsilon_{j} y^2_j} $, the integral becomes 
\begin{align*}
& \left(\frac{9\epsilon_{j}  }{4}\right)  \int_0^{\frac{16}{9\epsilon_{j} M}}   \left(e^{-(k_{j}-1) x_{j}}\right)         \left(\frac{9\epsilon_{j} x_{j} }{8}\right)^{k_j-3} \, dx_{j} \\
& \leq \left(\frac{9\epsilon_{j}  }{4}\right) \int_0^{\frac{16}{9\epsilon_j M}}   e^{-(k_j -1) x_j}         \left(\frac{2 }{M}\right)^{k_j - 3}  \,dx_{j} \\
& = \left(\frac{9\epsilon_{j}  }{4}\right) \left(\frac{2 }{M}\right)^{k_j - 3}\int_0^{\frac{16}{9 \epsilon_j M}}   e^{-(k_j - 1) x_{j}}      \,dx_{j} \\ 
& =\left(\frac{9\epsilon_j  }{4}\right)  \left(\frac{2 }{M}\right)^{k_j-3}    
\left(  \frac{e^{-(k_j - 1)\frac{16}{9\epsilon_j M}}}{-(k_j - 1)}   + \frac{1}{k_j - 1} \right).
\end{align*}

Finally, putting together equations \eqref{case3a} and \eqref{alp}, we obtain 
\begin{equation}\label{a_jeq1}
 \sum_{s\in A'_{i,h} }    \prod_{j=1}^r \left| J_{k_j-1} \Big(  \frac{4 \pi\sqrt{\sigma_j (\eta_i u m_1m_2 )}}{|\sigma_j(s)|}\Big)\right| \ll   \prod_{j \in \mathfrak{I}_0} (k_j - 1)^{-\frac{1}{3}}  \prod_{j \in \mathfrak{I}_1} (k_j - 1)^{-\frac{1}{2}} \prod_{j \in \mathfrak{I}_2} \left(\frac{2 e}{M}\right)^{(k_j - 3)}    (k_j - 1)^{-\frac{4}{3}},    
\end{equation}
in the case $\mathfrak{I}_1 \neq \emptyset$.

\subsubsection{Case: $\mathfrak{I}_1 = \emptyset$}\label{step b}
In this case,  we use the bounds in \S \ref{subsec:bounds-I0-I2} and the arguments above to obtain
 \begin{align}\label{a_jneq1}
\sum_{s\in A'_{i,h} }    \prod_{j=1}^r \left| J_{k_j-1} \Big(  \frac{4 \pi\sqrt{\sigma_j (\eta_i u m_1m_2 )}}{|\sigma_j(s)|}\Big)\right|
\ll    \prod_{\mathfrak{I}_2} \left(\frac{2 e}{M}\right)^{(k_j-3)}  (k_j-1)^{-\frac{4}{3}} \prod_{\mathfrak{I}_0} (k_j-1)^{-\frac{1}{3}}.
\end{align}

\subsection{Putting everything together.}\label{last step}

We now use \eqref{eqn:Ai-prime-Ai-prime-h} and the bounds obtained above to bound the sum over $s \in A^\prime_i$. 

For any $1 \leq \rho \leq r$, suppose that $h \in [3^{\rho-1},3^{\rho}-1]$. The tuple corresponding to $h = 3^{\rho-1}$ is $(\delta_{j\rho})$. In this case $\mathfrak{I}_1 = \{\rho\}$ and $\mathfrak{I}_0 = \{1, \dots, r\} \setminus \{\rho\}$. The contribution of this $h$ to the sum is bounded above by 
$$ 
(k_{\rho}-1)^{-\frac{1}{2}}\prod_{j \neq \rho}(k_{j}-1)^{-\frac{1}{3}}.
$$ 
Now, consider $ h \in (3^{\rho-1},3^{\rho}-1],$ and its corresponding $r$-tuple $(a_1,\, \dots \,,a_r)$. We have $a_j=0$ for all $j>\rho$. For $j\leq \rho$, either $a_j=0$ for all $j< \rho$ (and in that case $a_\rho = 2$), or $a_j=1$ or $2$ for some $j<\rho$. In the former case, the contribution of this $h$ to the sum is bounded above by  
$$ 
\left(\frac{2 e}{M}\right)^{(k_{\rho}-3)}  (k_{\rho}-1)^{-\frac{4}{3}}\prod_{j\neq \rho}(k_{j}-1)^{-\frac{1}{3}}.
$$ 
In the latter case, contribution of this $h$ to the sum is bounded above by 
$$ 
\o\left((k_{\rho}-1)^{-\frac{1}{2}}\prod_{'\neq \rho}(k_{j}-1)^{-\frac{1}{3}}\right).
$$
Therefore, for any $h \in [3^{\rho-1},3^{\rho}-1]$, the contribution of $h$ to the sum is bounded above by  
$$ 
(k_{\rho}-1)^{-\frac{1}{2}}\prod_{j\neq \rho}(k_{j}-1)^{-\frac{1}{3}},
$$
and the dominant terms are when $h = 3^{\rho - 1}$. Varying $h$ from 1 to $3^r-1$ and collecting dominant terms, we obtain
\begin{align}\label{sinAip}
\sum_{s\in A_i' }    \prod_{j=1}^r\left| J_{k_j-1} \Big(  \frac{4 \pi\sqrt{\sigma_j (\eta_i u m_1m_2 )}}{|\sigma_j(s)|}\Big)\right|
\ll \sum_{\rho = 1}^r \left( (k_{\rho}-1)^{-\frac{1}{2}}\prod_{j \neq \rho}(k_{j}-1)^{-\frac{1}{3}}\right).
\end{align}

By combining equations \eqref{Ai'raw} and \eqref{sinAip}, for fixed $i$ and $u\in U$  with $\eta_i u\in F^+$, we have 
 $$ 
\sum_{s\in  \mathfrak{\mathtt{b}_i} \mathfrak{N}/\pm, s\neq 0 }\omega_{\mathrm{f}}(s \mathtt{b}_i^{-1} ) S_{\omega_\mathfrak{N}} (m_1,m_2;\eta_i u \mathtt{b}_i^{-2};s\mathtt{b}_i^{-1})
 \frac{\sqrt{\Nr(\eta_i u)}}{\Nr(s)}
 \prod_{j=1}^r\frac{2\pi}{(\sqrt{-1})^{k_j}}J_{k_j-1} \Big(  \frac{4 \pi\sqrt{\sigma_j (\eta_i u m_1m_2 )}}{|\sigma_j(s)|}\Big) 
    $$
  $$ 
= \sum_{s\in A_i }\omega_{\mathrm{f}}(s \mathtt{b}_i^{-1} ) S_{\omega_\mathfrak{N}} (m_1,m_2;\eta_i u \mathtt{b}_i^{-2};s\mathtt{b}_i^{-1})
 \frac{\sqrt{\Nr(\eta_i u)}}{\Nr(s)}   
 \prod_{j=1}^r\frac{2\pi}{(\sqrt{-1})^{k_j}}J_{k_j-1} \Big(  \frac{4 \pi\sqrt{\sigma_j (\eta_i u m_1m_2 )}}{|\sigma_j(s)|}\Big)
$$
$$
+\O \left( \sum_{\rho=1}^r \left( (k_{\rho}-1)^{-\frac{1}{2}}\prod_{j\neq \rho}(k_{j}-1)^{-\frac{1}{3}}\right)\right) . 
$$    
By varying $i$ and $u\in U$ with $\eta_i u\in F^+$, we obtain the following refinement of \eqref{PTF}:
$$ \frac{e^{2\pi tr_{\mathbb{Q}}^F (m_1+m_2)}}{{\psi(\mathfrak{N})}}\Bigg[\prod_{j=1}^r  \frac{(k_j-2)!}{(4\pi \sqrt{\sigma_j(m_1m_2)})^{k_j-1}}\Bigg]
\sum_{{\phi}\in \mathcal{F} }\frac{\lambda_\mathfrak{n}^\phi W_{m_1}^\phi(1) \overline{W_{m_2}^\phi(1)}}{\|\phi\|^2}
=\, \Hat{T}(m_1,m_2,\mathfrak{n})\frac{\sqrt{d_F\Nr(\mathfrak{n})}}{\omega_\mathfrak{N}(m_1/s)\omega_{\mathrm{f}}(s)}
 $$   
  $$ 
+\sum_{i=1}^t \sum_{u\in U, \eta_i u\in F^+}\sum_{s\in A_i }\Bigg\{ \omega_{\mathrm{f}}(s\mathtt{b}_i^{-1} ) S_{\omega_\mathfrak{N}} (m_1,m_2;\eta_i u \b_i^{-2};s\b_i^{-1})
 $$
 $$ 
\frac{\sqrt{\Nr(\eta_i u)}}{\Nr(s)}\times \prod_{j=1}^r\frac{2\pi}{(\sqrt{-1})^{k_j}}J_{k_j-1} \Big(  \frac{4 \pi\sqrt{\sigma_j (\eta_i u m_1m_2 )}}{|\sigma_j(s)|}\Big)
 \Bigg\}+
\O_{A_i,u,\eta_i} \left(\sum_{j=1}^r \left( (k_{j}-1)^{-\frac{1}{2}}\prod_{j'\neq j}(k_{j'}-1)^{-\frac{1}{3}}\right)\right).
$$
This completes the proof of Theorem \ref{Main theorem}. 
\end{proof}
\bigskip

\section{Main term in special cases: Proof of Theorem \ref{Single term}}\label{sec: special-cases}

This section considers the main term in Theorem \ref{Main theorem}, namely, the sum 
$$
\, \Hat{T}(m_1,m_2,\mathfrak{n})\frac{\sqrt{d_F\Nr(\mathfrak{n})}}{\omega_\mathfrak{N}(m_1/s)\omega_{\mathrm{f}}(s)}
+\sum_{i=1}^t \sum_{u\in U, \eta_i u\in F^+}\sum_{s\in A_i }\Bigg\{ \omega_{\mathrm{f}}(s\b_i^{-1} ) S_{\omega_\mathfrak{N}} (m_1,m_2;\eta_i u \b_i^{-2};s\b_i^{-1})
 $$
 $$ 
\frac{\sqrt{\Nr(\eta_i u)}}{\Nr(s)}\times \prod_{j=1}^r\frac{2\pi}{(\sqrt{-1})^{k_j}}J_{k_j-1} \Big(  \frac{4 \pi\sqrt{\sigma_j (\eta_i u m_1m_2 )}}{|\sigma_j(s)|}\Big)
 \Bigg\}, 
 $$ 
 in special cases.  We wish to obtain a lower bound for this, which involves a triple sum. We are able to do this in special cases when the triple sum simplifies to a single term.  

We now assume that the field $F$ has odd narrow class number.  Note that, for any abelian group $A$, the quotient group $A/A^2$ is $2$-torsion. If the quotient is a finite group, then its cardinality is  $|A/A^2|=2^a$ and we denote this exponent as $\dim_2(A):=a$. If the group is infinite, then we denote $\dim_2 (A) = \infty$. Recall that we have $U={\mathcal{O}^\times}/{\mathcal{O}^\times}^2$ and we let $U^+=U\cap F^+$.  We also denote the narrow class group of $F$ by $C^+$. Recall that the narrow class group is the set of equivalence classes of fractional ideals up to totally positive principal ideals.

Recall that $t$ denotes the number of solutions to the equation $ [\mathfrak{b}]^2[\mathfrak{n}] = 1$ in the ideal class group.
 \begin{lemma}[{{\cite[Prop.\ 2.4]{EMP}}}]\label{t=1}
 If $F$ has odd narrow class number, then $t=1$ and $|U^+|=1.$
 \end{lemma} 
 \begin{proof}
 Since the class number is a divisor of the narrow class number, the class number is odd. As in \cite[Example 5.16]{KL}, the equation $[\mathfrak{b}]^2[\mathfrak{n}]=1$ has a unique solution, 
which implies $t=1$.
By \cite[Proposition 2.4]{EMP}, we  have $\dim_2(U^+)\leq \dim_2(C^+)$. Since, the latter is $0$, this implies that $|U^+|=1$.
 \end{proof}
Lemma \ref{t=1} implies that we may take $\eta_1=1$,  and $|\{u\in U \, : \, \eta_i u\in F^+$ for some  $i=1,\dots, t\}|=1$. 
Note that the set $A_1$ still need not be a singleton set.
In the next few lemmas, we illustrate some cases in which  $|A_1|=1$ and in those cases $\sum_{s\in A_1}$ is a single term.
 \begin{lemma}\label{A1}
Suppose the $\inf\{\|\sigma(s)\|  \, : \, s\in \mathfrak{b}_1\mathfrak{N}/\pm, s\neq 0 \}=\delta_1$ is attained for some $s_0$ such that  $\sigma(s_0)=(a,a,\,...\hspace{1mm},\hspace{1mm}a)=\Big(\frac{\delta_1}{\sqrt{r}},\frac{\delta_1}{\sqrt{r}},\,...\hspace{1mm},\hspace{1mm}\frac{\delta_1}{\sqrt{r}}\Big). $   Then  $A_1=\Big\{\sigma(s_0) \Big\}.$ 
\end{lemma}
\begin{proof}
Let $s\in A_1$. We have $|\sigma_j(s)|\leq \frac{\delta_1}{\sqrt{r}}$ for all $j$. Hence  $A_1\subset S_{\delta_1},$ where 
$$
S_{\delta_1}=\{ s\in \mathfrak{b}_1\mathfrak{N}/\pm \, : \, \sigma^2_1(s)+ \cdots +\sigma_r^2(s)\leq \delta^2_1\}.
$$ 
Let $s'\in A_1\cap S_{\delta_1}. $ By the choice of $\delta_1$, we have  $s'\in \partial A_1\cap S_{\delta_1}$ and hence $|\sigma_j(s')|=\frac{\delta_1}{\sqrt{r}} $ for all $j=1,...,r.$
Therefore $s'$ is equal to  $  \left((-1)^{m_1}\frac{\delta_1}{\sqrt{r}},(-1)^{m_2}\frac{\delta_1}{\sqrt{r}},\,...\hspace{1mm},\hspace{1mm}(-1)^{m_r}\frac{\delta_1}{\sqrt{r}}\right)$    with $m_j=0,1$ for $j=1,...\,,r$.
If  $m_j=0$ or $m_j=1$ for all $j$, then $s'=s$. In any other case, we have $m_{j_0}=0$ for some $j_0$, so that $\sigma_{j_0}(s)=\sigma_{j_0}(s')$.  However, $\sigma_{j_0}$  injective implies $s'=s_0=\frac{\delta_1}{\sqrt{r}}$.
\end{proof}
\begin{lemma}\label{deltaA_1}
Let $\mathfrak{b}_1\mathfrak{N}=\mathcal{O}$. Then $\delta_1=\sqrt{r}$ and $A_1=\{\sigma(1)\}.$
\end{lemma}
\begin{proof}
We wish to minimize $\inf\{\|\sigma(s)\| \, : \, s\in \mathcal{O}/\pm, s\neq 0 \}.$ This is equivalent to minimize $\| \sigma(s)\|^2 =\sigma_1^2(s)+ \cdots +\sigma_r^2(s). $ For a given $s$, let
$$\sigma(s')=\left((-1)^{m_1}\sigma_1(s),(-1)^{m_2}\sigma_2(s),\,...\hspace{1mm},\hspace{1mm}(-1)^{m_r}\sigma_r(s)\right).$$ 
Then, $\|\sigma(s')\|=\|\sigma(s)\|. $ Thus, without loss of generality, we can consider minimizing $\|\sigma(s)\|^2$ on the set $\{s\, : \, s\in\mathcal{O}/\pm, s\neq 0, \sigma_j(s)\geq 0$ for $j=1,...\,,r$$\}.$ Using the Cauchy-Schwartz inequality, we have 
$$ \sqrt{(1^2+...+1^2)( \sigma_1^2(s)+ ... +\sigma_r^2(s))}\geq \sigma_1(s)+...+\sigma_r(s). $$ 
Using the AM-GM inequality, $$ \frac{\sigma_1(s)+...+\sigma_r(s)}{r}\geq (\sigma_1(s)\times...\times \sigma_r(s))^{\frac{1}{r}}=\Nr(s)^{\frac{1}{r}}.$$
The equality holds when $\sigma_1(s)= \cdots =\sigma_r(s). $ In such case $\sigma_1^2(s)+ ... +\sigma_r^2(s)=r\sigma_1^2(s)=r(\Nr(s))^{\frac{2}{r}}. $ As $\Nr(s)\in \mathbb{N},$ the minimum value of $r(\Nr(s))^{\frac{2}{r}}=r,$ which happens when $\Nr(s)=1$ and $\sigma_1^2(s)=1$ implying $\sigma_1(s)=1.$ Thus $\delta_1=\sqrt{r}$ and on applying Lemma \ref{A1} we have $A_1=\{\sigma(1)\}.$
\end{proof}
 \begin{lemma}\label{A1A1}
Let $\mathfrak{b}_1\mathfrak{N}$ be an ideal in $\mathcal{O}$  such that $\mathfrak{b}_1\mathfrak{N}=\tilde{s}\mathcal{O}$ with $\tilde{s}\in \mathbb{Z}.$  Then, $\delta_1=|\tilde{s}|\sqrt{r}$ and $A_1=\{\sigma(|\tilde{s}|)\}.$
\end{lemma}
\begin{proof}
We wish to minimize $\inf\{\|\sigma(s)\|  \, : \, s\in \mathfrak{b}_1\mathfrak{N}/\pm, s\neq 0 \}.$ Without loss of generality this is equivalent to minimizing $\|\sigma(s)\|^2 =\sigma_1^2(s)+ ... +\sigma_r^2(s)  $ on the set 
 $$
 B=\{s\in \mathfrak{b}_1\mathfrak{N}/\pm :\,s\neq 0,\,\sigma_j(s)\geq 0 \textrm{ for } j=1,..,r \}.
 $$ 
 Proceeding similar to the argument in Lemma \ref{deltaA_1}, the minimum possible value of $\|\sigma(s)\|^2$ is  $r(\Nr(s))^{\frac{2}{r}}. $ 
 For $s\in B, $  there exists $s'\in \mathcal{O}$ such that $s=\tilde{s}s'$. Since norm is multiplicative, the  minimum value for $r(\Nr(s))^{\frac{2}{r}} $ is attained at $\tilde s$. This shows $\| \sigma(s) \|^2$ attains its minimum value $r(\Nr(\tilde{s}))^{\frac{2}{r}}$ at $\sigma(\tilde s)=(|\tilde{s}|,\,...\,,|\tilde{s}|).$ By Lemma \ref{A1}, we get $A_1=\{\sigma(|\tilde{s}|)\}.$
\end{proof}
\begin{remark}
It must be noted that Lemma \ref{A1A1} is a sufficient condition to have $|A_1|=1.$
Therefore in the hypothesis of  Theorem \ref{Single term}, the condition that $\mathfrak{b}_1\mathfrak{N}=\tilde{s}\mathcal{O} $ with $\tilde{s}\in \mathbb{Z}$   can be replaced  with the condition that $|\{s \in \mathfrak{b}_1\mathfrak{N}/\pm \, : \, s\neq 0, \|\sigma(s)\|=\delta_1\}|=1$.
\end{remark}

We now consider a sufficient condition for the non-vanishing of certain Kloosterman sums that helps us to prove Theorem \ref{Main theorem 2}. 
 \begin{lemma}\label{Kloosterman sum nonzero}
 Let  $\mathfrak{N}=\tilde{s}\mathcal{O}$ be unramified with $\tilde{s}\in \mathbb{N},$  $\tilde{s}$ be squarefree. Let $\tilde{m}_1,\tilde{m}_2 \in \mathfrak{d}^{-1}_+$ and  $\omega_{\mathfrak{N}}$  be trivial. Then   $$S_{\omega_\mathfrak{N}} (\tilde{m}_1,\tilde{m}_2;1;\tilde{s})\neq 0.$$
 \end{lemma}
\begin{proof}
We have $$S_{\omega_\mathfrak{N}} (\tilde{m}_1,\tilde{m}_2;1;\tilde{s})=\prod _{\nu<\infty } S_{\omega_{\mathfrak{N},\nu}} (\tilde{m}_{1\nu},\tilde{m}_{2\nu};1;\tilde{s}_\nu).$$ Let $\tilde{s}\mathcal{O}=\prod_{l=1}^{t'}\mathfrak{p}_l$ for distinct prime ideals $\mathfrak{p}_l.$  Let $\nu_l$ be the corresponding valuation for the prime ideal $\mathfrak{p}_l.$ Hence 
$$
\prod _{\nu<\infty } S_{\omega_{\mathfrak{N},\nu}} (\tilde{m}_{1\nu},\tilde{m}_{2\nu};1;\tilde{s}_\nu)=\prod_{l=1}^{t'} S_{\omega_{\mathfrak{N},\nu_l}} (m_{1\nu_l},m_{2\nu_l};1;\varpi_{\nu_l})
$$
where $\varpi_{\nu_l}$ is a generator of the maximal ideal $(\mathfrak{p}_l)_{\nu_l}=\mathfrak{p}_l\mathcal{O}_{\nu_l}$ and   $m_{1\nu_l},m_{2\nu_l} \in  \hat{\mathfrak{d}}_{\nu_l}^{-1}.$ 
For trivial $\omega_\mathfrak{N},$ we have 
\begin{align*}
S_{\omega_{\mathfrak{N},\nu_l}} (m_{1\nu_l},m_{2\nu_l};1;\varpi_{\nu_l})&=\sum_{\substack{ s_{1},s_{2}\in \mathcal{O}_{\nu_l} / \varpi_{\nu_l} \mathcal{O}_{\nu_l}  \\{s_1s_2\equiv 1 \pmod {\varpi_{\nu_l}\mathcal{O}_{\nu_l}} }} }\theta_{\nu_l} \Big(\frac{m_{1\nu_l} s_{1}+m_{2\nu_l} s_{2}}{\varpi_{\nu_l}}\Big) \\
&=\sum_{\substack{ s_{1},s_{2}\in \mathcal{O}_{\nu_l} / \varpi_{\nu_l} \mathcal{O}_{\nu_l}  \\{s_1s_2\equiv 1 \pmod {\varpi_{\nu_l}\mathcal{O}_{\nu_l}} }} } e\Bigg(\text{Tr}\Big( \frac{m_{1\nu_l} s_{1}+m_{2\nu_l} s_{2}}{\varpi_{\nu_l}}  \Big)\Bigg),
\end{align*}
where $e(x)=e^{2\pi i x}.$
 Let $p_l=\mathfrak{p}_l \cap \mathbb{Z}$, and we have 
\begin{align*}
\Bigg[e\Bigg(\mathrm{Tr}\Big( \frac{m_{1\nu_l} s_{1}+m_{2\nu_l} s_{2}}{\varpi_{\nu_l}}  \Big)\Bigg)\Bigg]^{p_l} &= 
e\Bigg( \mathrm{Tr}\Big( p_l \cdot  \frac{m_{1\nu_l} s_{1}+m_{2\nu_l} s_{2}}{\varpi_{\nu_l}}  \Big)\Bigg) \\
&= e\Bigg( \mathrm{Tr}\Big(   \frac{m_{1  \nu_l}p_ls_{1}}{\varpi_{\nu_l}}  \Big)\Bigg)\cdot e\Bigg( \mathrm{Tr}\Big(   \frac{m_{2  \nu_l}p_ls_{2}}{\varpi_{\nu_l}}  \Big)\Bigg)\\
&= 1.
\end{align*}
The last equality follows from the fact that $ \frac{m_{1  \nu_l}p_ls_{1}}{\varpi_{\nu_l}} $ and $   \frac{m_{2  \nu_l}p_ls_{2}}{\varpi_{\nu_l}}$ belong to the local inverse different $ \hat{\mathfrak{d}}_{\nu_l}^{-1}$.

 Suppose $S_{\omega_{\mathfrak{N},\nu_l}} (m_{1\nu_l},m_{2\nu_l};1;\varpi_{\nu_l})= 0.$ 
Note that $\mathbb{Z}\Big[e^{\frac{2\pi i}{p_l}}\Big]$ is isomorphic to $\mathbb{Z}[x]/(\Phi_{p_l}(x))$ where  $e^{\frac{2\pi i}{p_l}}$ gets mapped to $x+(\Phi_{p_l}(x))$ and $\Phi_{p_l}(x)=1+x+x^2+\dots+x^{p_l-1}. $ We further have
$\mathbb{Z}[x]/(\Phi_{p_l}(x),p_l)=\mathbb{Z}[x]/((x-1)^{p_l-1},p_l)$. Consider the ring homomorphism  $$
\mathbb{Z}[x]/(\Phi_{p_l}(x))\rightarrow \mathbb{Z}[x]/((x-1)^{p_l-1},p_l) \rightarrow \mathbb{Z}[x]/((x-1),p_l)\rightarrow \mathbb{F}_{p_l},
$$
where $$x+(\Phi_{p_l}(x))\mapsto x+((x-1)^{p_l-1},p_l) \mapsto x+((x-1),p_l)$$ $$=1+[x-1]+((x-1),p_l)=1+((x-1),p_l)\mapsto 1.$$
This implies that $e^{\frac{2\pi i}{p_l}}$ gets mapped to 1 via the ring homomorphism.
Now  $e\Bigg(\mathrm{Tr}\Big( \frac{m_{1\nu_l} s_{1}+m_{2\nu_l} s_{2}}{\varpi_{\nu_l}}  \Big)\Bigg)$ lies in  $\mathbb{Z}\Big[e^{\frac{2\pi i}{p_l}}\Big]$ so that $e\Bigg(\mathrm{Tr}\Big( \frac{m_{1\nu_l} s_{1}+m_{2\nu_l} s_{2}}{\varpi_{\nu_l}}  \Big)\Bigg)$ gets mapped to $1$ via the ring homomorphism. This implies  $S_{\omega_{\mathfrak{N},\nu_l}} (m_{1\nu_l},m_{2\nu_l};1;\varpi_{\nu_l})$ will get mapped to $p_l^f-1=-1 \in \mathbb{F}_{p_l},$ where 
$\big|\mathcal{O}_{\nu_l} / \varpi_{\nu_l} \mathcal{O}_{\nu_l}\big|=p_l^f$ for some natural number $f.$ This is a contradiction since image of  $S_{\omega_{\mathfrak{N},\nu_l}} (m_{1\nu_l},m_{2\nu_l};1;\varpi_{\nu_l}) $ should be $0.$
\end{proof}

 \begin{corollary}\label{lower bound}
 Let $F$ have odd narrow class number and suppose that the assumptions of Theorem \ref{Main theorem} hold true. Let $\mathfrak{b}_1\mathfrak{N}=\tilde{s}\mathcal{O} $ with $\tilde{s}\in \mathbb{Z}$. As $k_0\rightarrow \infty$,  assume that $S_{\omega_\mathfrak{N}} ( m_{1}, m_{2};\eta_1 \b_1^{-2};s\b_1^{-1})\neq 0$ for all $ m_{1}$ and $ m_{2}$. Then 
  \begin{align*}
  \Bigg| \frac{e^{2\pi tr_{\mathbb{Q}}^F ( m_{1}+ m_{2})}}{{\psi(\mathfrak{N})}}\Bigg[\prod_{j=1}^r  \frac{(k_j-2)!}{(4\pi \sqrt{\sigma_j( m_{1} m_{2})})^{k_j-1}}\Bigg] &
\sum_{{\phi}\in \mathcal{F} }\frac{\lambda_\mathfrak{n}^\phi W_{ m_{1}}^\phi(1) \overline{W_{ m_{2}}^\phi(1)}}{\| \phi \|^2} - \Hat{T}( m_{1}, m_{2},\mathfrak{n})\frac{\sqrt{d_F\Nr(\mathfrak{n})}}{\omega_\mathfrak{N}( m_{1}/s)\omega_{\mathrm{f}}(s)}\Bigg| \\ 
&\gg \prod_{j=1}^r (k_j-1)^{-\frac{1}{3}}
 \end{align*} 
 \end{corollary}
 \begin{proof}
 Using Theorem \ref{Main theorem}, Lemma \ref{t=1} and Lemma \ref{A1A1},  we have 
 \begin{align*}
&\frac{e^{2\pi tr_{\mathbb{Q}}^F ( m_{1}+ m_{2})}}{{\psi(\mathfrak{N})}}\Bigg[\prod_{j=1}^r  \frac{(k_j-2)!}{(4\pi \sqrt{\sigma_j( m_{1} m_{2})})^{k_j-1}}\Bigg]
\sum_{{\phi}\in \mathcal{F} }\frac{\lambda_\mathfrak{n}^\phi W_{ m_{1}}^\phi(1) \overline{W_{ m_{2}}^\phi(1)}}{\| \phi \|^2} - \Hat{T}( m_{1}, m_{2},\mathfrak{n})\frac{\sqrt{d_F\Nr(\mathfrak{n})}}{\omega_\mathfrak{N}( m_{1}/s)\omega_{\mathrm{f}}(s)}
\\&= \omega_{\mathrm{f}}(s\b_1^{-1} ) S_{\omega_\mathfrak{N}} ( m_{1}, m_{2};\eta_1 \b_1^{-2};s\b_1^{-1})
\frac{\sqrt{\Nr(\eta_1)}}{\Nr(s)}\times \prod_{j=1}^r\frac{2\pi}{(\sqrt{-1})^{k_j}}J_{k_j-1} \Big(  \frac{4 \pi\sqrt{\sigma_j (\eta_1  m_{1} m_{2} )}}{|\sigma_j(s)|}\Big)\\
& \hspace{1cm} +\o\Big(\prod_{j=1}^r \big(k_j-1\big)^{-\frac{1}{3}}\Big).
\end{align*}
Note that 
 $$\frac{4 \pi\sqrt{\sigma_j (\eta_1  m_{1} m_{2} )}}{|\sigma_j(s)|}=\frac{4 \pi\sqrt{\sigma_j (\eta_1  m_{1} m_{2} )}}{(\delta_1/\sqrt{r})}=\frac{2 \pi\sqrt{\sigma_j (\eta_1  m_{1} m_{2} )}}{\delta}.
$$
By the given condition 
$$\frac{2\pi \gamma_j\sqrt{\sigma_{j}( m_{1} m_{2})}}{\delta}\in \Big((k_{j}  -1)-(k_{j}  -1)^{\frac{1}{3}},\big(k_{j}  -1)\Big), $$ we can write
$$
\frac{2\pi \gamma_j\sqrt{\sigma_{j}( m_{1} m_{2})}}{\delta}=(k_{j}  -1)+d (k_{j}  -1)^{\frac{1}{3}}
$$
with $d\in (-1,0). $
Using Lemma \ref{Bessel-estimate}, part \ref{Bessel-estimate-part-4} we have
$$
\prod_{j=1}^rJ_{k_j-1} \Big(  \frac{4 \pi\sqrt{\sigma_j (\eta_1  m_{1} m_{2} )}}{|\sigma_j(s)|}\Big)\gg \prod_{j=1}^r \big(k_j-1\big)^{-\frac{1}{3}}
$$
and this completes the proof of this corollary. 
 \end{proof}

 \begin{proof}[Proof of Theorem \ref{Single term}]
 The theorem follows from Lemmas \ref{t=1}, \ref{A1A1} and Corollary \ref{lower bound}.
 \end{proof}
 
\begin{remark}
Note that the assumption $S_{\omega_\mathfrak{N}} ( m_{1}, m_{2};\eta_1 \b_1^{-2};s\b_1^{-1})\neq 0$ is essential to have a lower bound like Corollary \ref{lower bound}. For instance the main term of Theorem 1.7 of \cite{JS} is $$J_{k-1}(4\pi\sqrt{mn})\frac{\mu(N)}{N} \prod_{p|N} \big(1-p^{-2}\big).$$
This has a lower bound of $\frac{1}{k^{\frac{1}{3}}}$ for squarefree levels as we have $\mu(N)\neq 0$ if and only if $N$ is squarefree.  Hence, the non-vanishing of the Kloosterman sum involved plays an important role.
\end{remark}

\section{Lower bounds for discrepancy between measures: Proof of Theorem \ref{Main theorem 2}}\label{sec: proof-Main-theorem-2}

In order to prove Theorem \ref{Main theorem 2}, we first obtain a result  similar to Corollary \ref{lower bound} under the assumption that  
$$\frac{2\pi \gamma_j\sqrt{\sigma_{j}(m_{1}m_{2})}}{\delta}\in \Big((k_{j}  -1)-(k_{j}  -1)^{\frac{1}{3}},(k_{j}  -1)\Big),$$ for all $j,$ where $\gamma_j=\frac{\sqrt{\sigma_j(\eta_1)}}{|\sigma_j(d)|}$. 

\begin{lemma}\label{Main lemma}
 Let $F$ have odd narrow class number,  $\mathfrak{b}_1\mathfrak{N}=\tilde{s}\mathcal{O}$ with $\tilde{s}\in \mathbb{Z}$  and  $  \mathfrak{d}=d \mathcal{O} $.  
As $k_0 \rightarrow \infty,$ for $k=(k_1,\dots, k_r),$ suppose 
$\frac{2\pi \gamma_j\sqrt{\sigma_{j}( m_{1} m_{2})}}{\delta}\in \Big((k_{j}  -1)-(k_{j}  -1)^{\frac{1}{3}},(k_{j}  -1)\Big)$ for all $j , $ where        $\gamma_j=\frac{\sqrt{\sigma_j(\eta_1)}}{|\sigma_j(d)|}.$ 
Further assume that  $S_{\omega_\mathfrak{N}} ( m_{1}, m_{2};\eta_1 \b_1^{-2};s\b_1^{-1})\neq 0$ for a sequence of  $ m_{1}$ and $ m_{2}$. Then 
 $$\Bigg| \frac{e^{2\pi tr_{\mathbb{Q}}^F ( m_{1}+ m_{2})}}{{\psi(\mathfrak{N})}}\Bigg[\prod_{j=1}^r  \frac{(k_j-2)!}{(4\pi \sqrt{\sigma_j( m_{1} m_{2})})^{k_j-1}}\Bigg]
\sum_{{\phi}\in \mathcal{F} }\frac{\lambda_\mathfrak{n}^\phi W_{ m_{1}}^\phi(1) \overline{W_{ m_{2}}^\phi(1)}}{\| \phi \|^2}
 $$
$$-\hspace{1mm}\Hat{T}( m_{1}, m_{2},\mathfrak{n})\frac{\sqrt{d_F\Nr(\mathfrak{n})}}{\omega_\mathfrak{N}( m_{1}/s)\omega_{\mathrm{f}}(s)}\Bigg| \gg\prod_{j=1}^r (k_j-1)^{-\frac{1}{3}}
 $$   
 \end{lemma}
\begin{proof}
To see this, change $\gamma_j=\frac{\sqrt{\sigma_j(\eta_1)}}{|\sigma_j(d)|}$ in the proofs of Theorem \ref{Main theorem} and Theorem \ref{Single term} in place of   $\gamma_j=\sqrt{\sigma_j(\eta_1)}$. The proof follows similarly.
\end{proof}
\begin{corollary}\label{Main lemma 2}
 Let $F$ has narrow class number equal to 1, $ \mathfrak{p}={p}\mathcal{O}$ and $ \mathfrak{d}=d\mathcal{O}$. Let  $\mathfrak{b}_1\mathfrak{N}=\tilde{s}\mathcal{O}$ with $\tilde{s}\in \mathbb{Z}$ and  $|\tilde{s}|$ being squarefree. Further, let $\omega_{\mathfrak{N}}$ be trivial and $\ell$ be odd such that 
 $$
 \frac{2\pi \gamma_j\sqrt{\sigma_{j}({p}^{\ell})}}{\delta |\sigma_j(d)|} \in \Big((k_{j}  -1)-(k_{j}  -1)^{\frac{1}{3}},(k_{j}  -1)\Big)
 $$ 
 for all $j ,$  where   $\gamma_j=\frac{\sqrt{\sigma_j(\eta_1)}}{|\sigma_j(d)|}.$ 
 Then 
$$ 
\Bigg|\frac{1}{\sqrt{\Nr({p}^{\ell})}}\Bigg[\prod_{j=1}^r  \frac{(k_{j}-2)!}{(4\pi)^{k_{j}-1} }
\sum_{{\phi}\in \mathcal{F} }\frac{\lambda^\phi_{\mathfrak{p}^\ell} }{\| \phi \|^2}\Bigg]\Bigg|
\gg \prod_{j=1}^r (k_{j}-1)^{-\frac{1}{3}}.
 $$
\end{corollary}
\begin{proof}
 Let us take $ m_{1}=\frac{{p}^{\ell}}{d}$ with  $\ell$ odd and $m_2=\frac{1}{d}$ in Corollary \ref{sec 2 corollary} to get 
\begin{align*}
& \frac{e^{4\pi r}  \Nr(d)}{\psi(\mathfrak{N})d_F^2\sqrt{\Nr({p}^{\ell})}}\Bigg[\prod_{j=1}^r  \frac{(k_j-2)!}{(4\pi)^{k_j-1} }\Bigg]
\sum_{{\phi}\in \mathcal{F} }\frac{\lambda^\phi_{\mathfrak{p}^{l}} }{\|\phi \|^2}
=\hspace{1mm}\Hat{T}\Big(\frac{{p}^{\ell}}{d},\frac{1}{d},\mathcal{O}\Big)\frac{\sqrt{d_F}}{\omega_\mathfrak{N}\Big(\frac{{p}^{\ell}}{sd}\Big)\omega_{\mathrm{f}}(s)}
\\&+\sum_{s\in \mathfrak{b}_1\mathfrak{N}/\pm, s\neq 0 }\Bigg\{\omega_{\mathrm{f}}(s\b_{1}^{-1}) S_{\omega_\mathfrak{N}} \Big(\frac{{p}^{\ell}}{d},\frac{1}{d};1;\tilde{s}\Big)\times
\frac{\sqrt{\Nr(\eta_1)}}{\Nr(s)}\times \prod_{j=1}^r\frac{2\pi}{(\sqrt{-1})^{k_j}}J_{k_j-1} \Big(  \frac{4 \pi\sqrt{\sigma_j (\eta_1 {p}^{\ell} )}}{|\sigma_j(sd)|}\Big)
 \Bigg\}. \end{align*}
 Since $\ell$ is odd, $\Hat{T}\Big(\frac{{p}^{\ell}}{d},\frac{1}{d},\mathcal{O}\Big)=0$ and 
Lemma \ref{Kloosterman sum nonzero}  implies that  $S_{\omega_\mathfrak{N}} \Big(\frac{{p}^{\ell}}{d},\frac{1}{d};1;\tilde{s}\Big) \neq 0.$ Applying   Lemma \ref{Main lemma}, we get  
$$\Bigg| \frac{e^{4\pi r}  \Nr(d)}{\psi(\mathfrak{N})d_F^2\sqrt{\Nr({p}^{\ell})}}\Bigg[\prod_{j=1}^r  \frac{(k_j-2)!}{(4\pi)^{k_j-1} }\Bigg]
\sum_{{\phi}\in \mathcal{F} }\frac{\lambda^\phi_{\mathfrak{p}^{\ell}} }{\| \phi \|^2}\Bigg|
  \gg  \prod_{j=1}^r (k_{j}-1)^{-\frac{1}{3}}.
 $$
 Thus 
 $$ \Bigg|\frac{1}{\sqrt{\Nr({p}^{\ell})}}\Bigg[\prod_{j=1}^r  \frac{(k_{j}-2)!}{(4\pi)^{k_{j}-1} }
\sum_{{\phi}\in \mathcal{F} }\frac{\lambda^\phi_{\mathfrak{p}^{\ell}} }{\|\phi\|^2}\Bigg]\Bigg|
\gg \prod_{j=1}^r (k_{j}-1)^{-\frac{1}{3}}.
 $$
\end{proof}
 Consider $\kappa^\phi_{\mathfrak{p}^\ell}=\frac{\lambda^\phi_{\mathfrak{p}^\ell}}{\sqrt{\Nr(\mathfrak{p}^\ell)}}$ and the discrete measure $$
\tilde{\nu}_{k,\mathfrak{N}}:=\prod_{j=1}^r  \frac{(4\pi)^{k_j-1}}{ (k_j-2)!}
\sum_{{\phi}\in \mathcal{F} }\frac{\delta_{\kappa^\phi_{\mathfrak{p}}} }{\|\phi \|^2}. 
$$ 

We now prove Theorem  \ref{Main theorem 2}.

\begin{proof}[Proof of Theorem \ref{Main theorem 2}]
Let $k_{lj}=\Big[\frac{2\pi \gamma_j\sqrt{\sigma_{j}({p}^l)}}{\delta|\sigma_j(d)|}\Big]-1,$ where $[x]$ denotes the greatest integer part of $x$.
It is easy to check that this $k_{lj} $ satisfies $\frac{2\pi \gamma_j\sqrt{\sigma_{j}({p}^l)}}{\delta|\sigma_j(d)|}\in \Big((k_{lj}  -1)-(k_{lj}  -1)^{\frac{1}{3}},(k_{lj}  -1)\Big).$  
This now satisfies the hypothesis in Corollary \ref{Main lemma 2} and we get 
\begin{equation}\label{lowerboundklj}
 \Bigg|\frac{1}{\sqrt{\Nr({p}^l)}} \prod_{j=1}^r  \frac{(k_{lj}-2)!}{(4\pi)^{k_{lj}-1} }
\sum_{{\phi}\in \mathcal{F} }\frac{\lambda^\phi_{\mathfrak{p}^l} }{\|\phi \|^2}  \Bigg|
\gg \prod_{j=1}^r (k_{lj}-1)^{-\frac{1}{3}}.
\end{equation}

We have (see, for example, Proposition 4.5 of \cite{KL})
$$\kappa^\phi_{\mathfrak{p}^l}=X_l(\kappa^\phi_{\mathfrak{p}})$$
where $X_l(2 \cos \theta)=\frac{\sin(l+1)\theta}{\sin \theta}$ is the degree $l$ Chebyshev polynomial of second kind. Hence
$$ \Bigg|\prod_{j=1}^r  \frac{(k_{l_j}-2)!}{(4\pi)^{k_{l_j}-1} }
\sum_{{\phi}\in \mathcal{F} }\frac{\kappa^\phi_{\mathfrak{p}^l} }{\|\phi \|^2}\Bigg|
\gg\prod_{j=1}^r (k_{l_j}-1)^{-\frac{1}{3}},
 $$
  equivalently,
$$ \Bigg|\prod_{j=1}^r  \frac{(k_{l_j}-2)!}{(4\pi)^{k_{l_j}-1} }
\sum_{{\phi}\in \mathcal{F} }\frac{X_l(\kappa^\phi_{\mathfrak{p}}) }{ \| \phi \|^2}\Bigg|
\gg \prod_{j=1}^r (k_{l_j}-1)^{-\frac{1}{3}}.
 $$
Therefore 
$$\Bigg|\int_{-2}^2 X_l(x) \, d\tilde{\nu}_{k_l,\mathfrak{N}}(x)\Bigg|   \gg \prod_{j=1}^r (k_{l_j}-1)^{-\frac{1}{3}}.
$$
The orthogonality of the polynomials $X_{l}(x)$ with respect to the Sato-Tate measure implies 
$$
\Bigg|\int_{-2}^2 X_{l}(x) \, d(\tilde{\nu}_{k_l,\mathfrak{N}}-\mu_\infty)(x) \Bigg| =\Bigg|\int_{-2}^2 X_{l}(x) \, d\tilde{\nu}_{k_l,\mathfrak{N}}(x)  -  \int_{-2}^2X_{l}(x) \, d\mu_\infty(x)\Bigg|$$
\begin{equation}\label{step 1}
=\Bigg|\int_{-2}^2 X_{l}(x) \, d\tilde{\nu}_{k_l,\mathfrak{N}}(x)\Bigg|\gg  \prod_{j=1}^r (k_{l_j}-1)^{-\frac{1}{3}}.
\end{equation}
Integration by parts and $\Big| X'_l(x)\Big|\ll l^2$  gives us 
\begin{equation}\label{step 2}
\Bigg|\int_{-2}^2 X_{l}(x) \, d(\tilde{\nu}_{k_l,\mathfrak{N}}-\mu_\infty)(x)\Bigg|
\ll  l^2\Bigg|  \int_{-2}^2  \, d(\tilde{\nu}_{k_l,\mathfrak{N}}-\mu_\infty)(x)  \Bigg|.
\end{equation}
Now 
\begin{align*} &D(\tilde{\nu}_{k_l,\mathfrak{N}},\mu_\infty)\geq \Bigg|\tilde{\nu}_{k_l,\mathfrak{N}}([-2,2])-\mu_\infty([-2,2])\Bigg| \\&= \Bigg|\int_{-2}^2 d(\tilde{\nu}_{k_l,\mathfrak{N}}-\mu_\infty)(x)\Bigg|\gg
\frac{1}{l^2} \Bigg|\int_{-2}^2 X_l(x) \, d(\tilde{\nu}_{k_l,\mathfrak{N}}-\mu_\infty)(x)\Bigg|,
\end{align*}
using equation \eqref{step 2}.
Equation \eqref{step 1}  implies
$$ 
D(\tilde{\nu}_{k_l,\mathfrak{N}},\mu_\infty)\gg\frac{1}{l^2\times\prod_{j=1}^r (k_{l_j}-1)^{\frac{1}{3}}}.
$$
However, $k_{l_j}=\Big[\frac{2\pi \gamma_j\sqrt{\sigma_{j}({p}^l)}}{|\sigma_j(d)|\delta}\Big]-1$ , so  $\frac{2\pi \gamma_j\sqrt{\sigma_{j}({p}^l)}}{|\sigma_j(d)|\delta}<2k_{l_j}. $ This implies $\frac{k_{l_j}\delta|\sigma_j(d)|}{\pi\gamma_j}>\big(\sqrt{\sigma_{j}({p})}\big)^l$ and hence $\log (k_{l})_0 \gg l$.
 Therefore 
 $$D(\tilde{\nu}_{k_l,\mathfrak{N}},\mu_\infty)\gg\frac{1}{\big(  \log (k_{l})_0 \big)^2  \times\prod_{j=1}^r (k_{l_j}-1)^{\frac{1}{3}}    },
 $$
which completes the proof. 
\end{proof}

\begin{remark}
Theorem \ref{Main theorem 2} generalizes Theorem 1.6  of \cite{JS} to ideals of rings of integers $\mathcal{O}$ with narrow class number $1$.   In particular, Theorem \ref{Main theorem 2} holds for the space $A_k(\mathcal{O},\omega)$  as $(k_{l})_0\rightarrow \infty $  with   trivial $\omega_\mathfrak{N}$.
\end{remark}
\bigskip

\subsection*{Acknowledgements}
The results in this article are contained in the second-named author's doctoral thesis.  The first named author acknowledges the support of ANRF grant MTR/2017/000114 in this project.  The second named author was supported by a Ph.D. fellowship from CSIR.  The third named author acknowledges the support of ANRF grants MTR/2019/001108 and CRG/2023/001743 in this project.  The second-named author thanks COEP Technological University, Pune, for providing an excellent working environment. \bigskip

\bibliographystyle{amsalpha}
\bibliography{biblio}

@article {Das2026,
    AUTHOR = {Das, Jishu},
     TITLE = {Shortest nonzero lattice points in totally real
              multi-quadratic number fields and applications},
   JOURNAL = {Int. J. Number Theory},
  FJOURNAL = {International Journal of Number Theory},
    VOLUME = {22},
      YEAR = {2026},
    NUMBER = {6},
     PAGES = {1251--1268},
      ISSN = {1793-0421,1793-7310},
   MRCLASS = {11F41 (11F60 11P21 11R21)},
  MRNUMBER = {5063650},
       DOI = {10.1142/S1793042126500673},
       URL = {https://doi.org/10.1142/S1793042126500673},
}

@article {OM,
    AUTHOR = {Omar, Sami and Mazhouda, Kamel},
     TITLE = {Equir\'{e}partition des coefficients de {F}ourier des
              fonctions {$L$} de carr\'{e}s et de cubes sym\'{e}triques},
   JOURNAL = {Ramanujan J.},
  FJOURNAL = {Ramanujan Journal. An International Journal Devoted to the
              Areas of Mathematics Influenced by Ramanujan},
    VOLUME = {20},
      YEAR = {2009},
    NUMBER = {1},
     PAGES = {81--89},
      ISSN = {1382-4090,1572-9303},
   MRCLASS = {11F25 (11F30 11F67 11M41 11N36)},
  MRNUMBER = {2546185},
MRREVIEWER = {Anton\ Deitmar},
       DOI = {10.1007/s11139-009-9157-1},
       URL = {https://doi.org/10.1007/s11139-009-9157-1},
}

@article {TW,
    AUTHOR = {Tang, Hengcai and Wang, Yingnan},
     TITLE = {Quantitative versions of the joint distributions of {H}ecke
              eigenvalues},
   JOURNAL = {J. Number Theory},
  FJOURNAL = {Journal of Number Theory},
    VOLUME = {169},
      YEAR = {2016},
     PAGES = {295--314},
      ISSN = {0022-314X,1096-1658},
   MRCLASS = {11F11 (11F25)},
  MRNUMBER = {3531241},
MRREVIEWER = {Thomas\ R.\ Shemanske},
       DOI = {10.1016/j.jnt.2016.05.011},
       URL = {https://doi.org/10.1016/j.jnt.2016.05.011},
}

@book {KL-traces,
    AUTHOR = {Knightly, Andrew and Li, Charles},
     TITLE = {Traces of {H}ecke operators},
    SERIES = {Mathematical Surveys and Monographs},
    VOLUME = {133},
 PUBLISHER = {American Mathematical Society, Providence, RI},
      YEAR = {2006},
     PAGES = {x+378},
      ISBN = {978-0-8218-3739-9; 0-8218-3739-7},
   MRCLASS = {11F72 (11-02 11F25 11F70)},
  MRNUMBER = {2273356},
MRREVIEWER = {Mark\ Rowland\ Budden},
       DOI = {10.1090/surv/133},
       URL = {https://doi.org/10.1090/surv/133},
}

@article {Li-2009,
    AUTHOR = {Li, Charles},
     TITLE = {On the distribution of {S}atake parameters of {${\rm GL}_2$}
              holomorphic cuspidal representations},
   JOURNAL = {Israel J. Math.},
  FJOURNAL = {Israel Journal of Mathematics},
    VOLUME = {169},
      YEAR = {2009},
     PAGES = {341--373},
      ISSN = {0021-2172,1565-8511},
   MRCLASS = {11F70 (11F72 22E55)},
  MRNUMBER = {2460909},
MRREVIEWER = {Jannis\ A.\ Antoniadis},
       DOI = {10.1007/s11856-009-0014-0},
       URL = {https://doi.org/10.1007/s11856-009-0014-0},
}

@book {Nguyen,
    AUTHOR = {Nguyen, Lan The},
     TITLE = {The {R}amanujan conjecture for {H}ilbert modular forms},
      NOTE = {Thesis (Ph.D.)--University of California, Los Angeles},
 PUBLISHER = {ProQuest LLC, Ann Arbor, MI},
      YEAR = {2005},
     PAGES = {82},
      ISBN = {978-0542-50874-5},
   MRCLASS = {Thesis},
  MRNUMBER = {2708305},
       URL =
              {http://gateway.proquest.com/openurl?url_ver=Z39.88-2004&rft_val_fmt=info:ofi/fmt:kev:mtx:dissertation&res_dat=xri:pqdiss&rft_dat=xri:pqdiss:3202774},
}

@incollection {Blasius,
    AUTHOR = {Blasius, Don},
     TITLE = {Hilbert modular forms and the {R}amanujan conjecture},
 BOOKTITLE = {Noncommutative geometry and number theory},
    SERIES = {Aspects Math., E37},
     PAGES = {35--56},
 PUBLISHER = {Friedr. Vieweg, Wiesbaden},
      YEAR = {2006},
      ISBN = {3-8348-0170-4},
   MRCLASS = {11F41 (11G18)},
  MRNUMBER = {2327298},
MRREVIEWER = {Fabrizio\ Andreatta},
       DOI = {10.1007/978-3-8348-0352-8\{_}2}

@article {Brylinski-Labesse,
    AUTHOR = {Brylinski, J.-L. and Labesse, J.-P.},
     TITLE = {Cohomologie d'intersection et fonctions {$L$} de certaines
              vari\'{e}t\'{e}s de {S}himura},
   JOURNAL = {Ann. Sci. \'{E}cole Norm. Sup. (4)},
  FJOURNAL = {Annales Scientifiques de l'\'{E}cole Normale Sup\'{e}rieure.
              Quatri\`eme S\'{e}rie},
    VOLUME = {17},
      YEAR = {1984},
    NUMBER = {3},
     PAGES = {361--412},
      ISSN = {0012-9593},
   MRCLASS = {11G18 (11F67 11F70 11G40 14G10 22E55)},
  MRNUMBER = {777375},
MRREVIEWER = {K.\ F.\ Lai},
       URL = {http://www.numdam.org/item?id=ASENS_1984_4_17_3_361_0},
}

@article {Deligne-Serre,
    AUTHOR = {Deligne, Pierre and Serre, Jean-Pierre},
     TITLE = {Formes modulaires de poids {$1$}},
   JOURNAL = {Ann. Sci. \'{E}cole Norm. Sup. (4)},
  FJOURNAL = {Annales Scientifiques de l'\'{E}cole Normale Sup\'{e}rieure.
              Quatri\`eme S\'{e}rie},
    VOLUME = {7},
      YEAR = {1974},
     PAGES = {507--530 (1975)},
      ISSN = {0012-9593},
   MRCLASS = {10D15 (12A65)},
  MRNUMBER = {379379},
MRREVIEWER = {Stephen\ Gelbart},
       URL = {http://www.numdam.org/item?id=ASENS_1974_4_7_4_507_0},
}

@article {Deligne,
    AUTHOR = {Deligne, Pierre},
     TITLE = {La conjecture de {W}eil. {I}},
   JOURNAL = {Inst. Hautes \'{E}tudes Sci. Publ. Math.},
  FJOURNAL = {Institut des Hautes \'{E}tudes Scientifiques. Publications
              Math\'{e}matiques},
    NUMBER = {43},
      YEAR = {1974},
     PAGES = {273--307},
      ISSN = {0073-8301,1618-1913},
   MRCLASS = {14G13},
  MRNUMBER = {340258},
MRREVIEWER = {Nicholas\ M.\ Katz},
       URL = {http://www.numdam.org/item?id=PMIHES_1974__43__273_0},
}

@article {JD,
    AUTHOR = {Das, Jishu},
     TITLE = {A lower bound for the discrepancy in a {S}ato–{T}ate type measure},
   JOURNAL = {Ramanujan J.},
  FJOURNAL = {Ramanujan Journal. An International Journal Devoted to the
              Areas of Mathematics Influenced by Ramanujan},
    VOLUME = {65},
      YEAR = {2024},
    NUMBER = {2},
     PAGES = {637 - 658},
      ISSN = {1382-4090,1572-9303},
       DOI = {10.1007/s11139-024-00909-3},
       URL = {https://doi.org/10.1007/s11139-024-00909-3},
}

@article {Royer,
    AUTHOR = {Royer, Emmanuel},
     TITLE = {Facteurs {$\bold Q$}-simples de {$J_0(N)$} de grande dimension
              et de grand rang},
   JOURNAL = {Bull. Soc. Math. France},
  FJOURNAL = {Bulletin de la Soci\'{e}t\'{e} Math\'{e}matique de France},
    VOLUME = {128},
      YEAR = {2000},
    NUMBER = {2},
     PAGES = {219--248},
      ISSN = {0037-9484},
   MRCLASS = {11G10 (11F30 11G18 11G40)},
  MRNUMBER = {1772442},
MRREVIEWER = {Andrea Mori},
       URL = {http://www.numdam.org/item?id=BSMF_2000__128_2_219_0},
}

@article {Li,
    AUTHOR = {Li, Charles C. C.},
     TITLE = {Kuznietsov trace formula and weighted distribution of {H}ecke
              eigenvalues},
   JOURNAL = {J. Number Theory},
  FJOURNAL = {Journal of Number Theory},
    VOLUME = {104},
      YEAR = {2004},
    NUMBER = {1},
     PAGES = {177--192},
      ISSN = {0022-314X},
   MRCLASS = {11F25},
  MRNUMBER = {2021634},
MRREVIEWER = {A. Raghuram},
       DOI = {10.1016/S0022-314X(03)00149-5},
       URL = {https://doi.org/10.1016/S0022-314X(03)00149-5},
}

@article {KL1,
    AUTHOR = {Knightly, A. and Li, C.},
     TITLE = {Kuznetsov's trace formula and the {H}ecke eigenvalues of
              {M}aass forms},
   JOURNAL = {Mem. Amer. Math. Soc.},
  FJOURNAL = {Memoirs of the American Mathematical Society},
    VOLUME = {224},
      YEAR = {2013},
    NUMBER = {1055},
     PAGES = {vi+132},
      ISSN = {0065-9266},
      ISBN = {978-0-8218-8744-8},
   MRCLASS = {11F72 (11F37 11F70 11L05)},
  MRNUMBER = {3099744},
MRREVIEWER = {\Dbar \cftil{o} Ng\d{o}c Di\cfudot{e}p},
       DOI = {10.1090/S0065-9266-2012-00673-3},
       URL = {https://doi.org/10.1090/S0065-9266-2012-00673-3},
}

@article {GJS,
    AUTHOR = {Gamburd, Alex and Jakobson, Dmitry and Sarnak, Peter},
     TITLE = {Spectra of elements in the group ring of {${\rm SU}(2)$}},
   JOURNAL = {J. Eur. Math. Soc. (JEMS)},
  FJOURNAL = {Journal of the European Mathematical Society (JEMS)},
    VOLUME = {1},
      YEAR = {1999},
    NUMBER = {1},
     PAGES = {51--85},
      ISSN = {1435-9855},
   MRCLASS = {11K99 (11M99 22E45 81Q50)},
  MRNUMBER = {1677685},
MRREVIEWER = {Jens Bolte},
       DOI = {10.1007/PL00011157},
       URL = {https://doi.org/10.1007/PL00011157},
}

@article {SZ,
    AUTHOR = {Sarnak, Peter and Zubrilina, Nina},
     TITLE = {Convergence to the {P}lancherel measure of {H}ecke
              eigenvalues},
   JOURNAL = {Acta Arith.},
  FJOURNAL = {Acta Arithmetica},
    VOLUME = {214},
      YEAR = {2024},
     PAGES = {191--213},
      ISSN = {0065-1036,1730-6264},
   MRCLASS = {11F11 (11F25)},
  MRNUMBER = {4772283},
MRREVIEWER = {Goran\ Djankovi\'{c}},
       DOI = {10.4064/aa230419-4-10},
       URL = {https://doi.org/10.4064/aa230419-4-10},
}

@article{EMP,
	author = {Edgar, H. M. and Mollin, R. A. and Peterson, B. L.},
	journal = {Proc. Amer. Math. Soc.},
	number = {1},
	pages = {33--37},
	title = {Class groups, totally positive units, and squares},
	volume = {98},
	year = {1986}}

@article{Golubeva,
	author = {Golubeva, E. P.},
	journal = {Zap. Nauchn. Sem. S.-Peterburg. Otdel. Mat. Inst. Steklov. (POMI)},
	number = {Anal. Teor. Chisel i Teor. Funkts. 20},
	pages = {33--40, 286},
	title = {Distribution of the eigenvalues of {H}ecke operators},
	volume = {314},
	year = {2004}}

@article{LJ,
	author = {Landau, L. J.},
	journal = {J. London Math. Soc. (2)},
	number = {1},
	pages = {197--215},
	title = {Bessel functions: monotonicity and bounds},
	volume = {61},
	year = {2000}}

@conference{KL,
	author = {Knightly, A. and Li, C.},
	booktitle = {Modular forms on Schiermonnikoog},
	pages = {145--187},
	publisher = {Cambridge Univ. Press, Cambridge},
	title = {Petersson's trace formula and the {H}ecke eigenvalues of {H}ilbert modular forms},
	year = {2008}}

@article{JS,
	author = {Jung, J. and Sardari, N. T.},
	journal = {Math. Ann.},
	pages = {513--557},
	title = {Asymptotic trace formula for the {H}ecke operators},
	volume = {378},
	year = {2020}}

@article{LW,
	author = {Lau, Y-K and Wang, Y.},
	journal = {J. Number Theory},
	pages = {2262--2281},
	title = {Quantitative version of the joint distribution of eigenvalues of the {H}ecke operators},
	volume = {131},
	year = {2011}}

@article{BGG,
	author = {T. Barnet-Lamb and Gee, T. and Geraghty, D.},
	journal = {J. Amer. Math. Soc.},
	number = {2},
	pages = {411--469},
	title = {The {S}ato-{T}ate conjecture for {H}ilbert modular forms},
	volume = {24},
	year = {2011}}

@incollection{Sarnak,
	author = {Sarnak, P.},
	booktitle = {Analytic number theory and {D}iophantine problems ({S}tillwater, {OK}, 1984)},
	journal = {1984)},
	pages = {321--331},
	publisher = {Birkh\"auser Boston, Boston, MA},
	series = {Progr. Math.},
	title = {Statistical properties of eigenvalues of the {H}ecke operators},
	volume = {70},
	year = {1987}}

@article{MS2,
	author = {Murty, M. R. and Sinha, K.},
	journal = {Proc. Amer. Math. Soc.},
	number = {10},
	pages = {3481-3494},
	title = {Factoring newparts of {J}acobians of certain modular curves},
	volume = {138},
	year = {2010}}

@article{N,
	author = {Nagoshi, H.},
	journal = {Proc. Amer. Math. Soc},
	month = {November},
	number = {11},
	pages = {3097-3106},
	title = {Distribution of {H}ecke eigenvalues},
	volume = {134},
	year = {2006}}

@article{CDF,
	author = {Conrey, B. and Duke, W. and Farmer, D.},
	journal = {Acta Arith.},
	number = {4},
	pages = {405-409},
	title = {The distribution of the eigenvalues of {H}ecke operators},
	volume = {78},
	year = {1997}}

@article{BGHT,
	author = {Barnet-Lamb, T. and Geraghty, D. and Harris, M. and Taylor, R.},
	doi = {10.2977/PRIMS/31},
	fjournal = {Publications of the Research Institute for Mathematical Sciences},
	issn = {0034-5318},
	journal = {Publ. Res. Inst. Math. Sci.},
	mrclass = {11F80 (11F11 11G18 14J32)},
	mrreviewer = {Neil P. Dummigan},
	number = {1},
	pages = {29--98},
	title = {A family of {C}alabi-{Y}au varieties and potential automorphy {II}},
	url = {http://dx.doi.org/10.2977/PRIMS/31},
	volume = {47},
	year = {2011}}

@article{LLW,
	author = {Lau, Y-K and Li, C. and Wang, Y.},
	doi = {10.4064/aa164-4-3},
	fjournal = {Acta Arithmetica},
	issn = {0065-1036},
	journal = {Acta Arith.},
	mrclass = {11F70 (11F72 22E55)},
	mrreviewer = {Anton Deitmar},
	number = {4},
	pages = {355--380},
	title = {Quantitative analysis of the {S}atake parameters of {${\rm GL}_2$} representations with prescribed local representations},
	url = {http://dx.doi.org/10.4064/aa164-4-3},
	volume = {164},
	year = {2014}}

@book{M,
	author = {Montgomery, H. L.},
	doi = {10.1090/cbms/084},
	isbn = {0-8218-0737-4},
	mrclass = {11-02 (11Kxx 11L07 11Mxx 11Nxx)},
	mrreviewer = {John B. Friedlander},
	pages = {xiv+220},
	publisher = {Published for the Conference Board of the Mathematical Sciences, Washington, DC; by the American Mathematical Society, Providence, RI},
	series = {CBMS Regional Conference Series in Mathematics},
	title = {Ten lectures on the interface between analytic number theory and harmonic analysis},
	url = {http://dx.doi.org/10.1090/cbms/084},
	volume = {84},
	year = {1994}}

@article{ALP,
	author = {{Aistleitner}, C. and {Lachmann}, T. and {Pausinger}, F.},
	journal = {J. Number Theory},
	month = {January},
	pages = {206-220},
	title = {{Pair correlations and equidistribution}},
	volume = {182},
	year = 2018}

@article{MSeffective,
	author = {Murty, M. R. and Sinha, K.},
	doi = {10.1016/j.jnt.2008.10.010},
	fjournal = {Journal of Number Theory},
	issn = {0022-314X},
	journal = {J. Number Theory},
	mrclass = {11F25 (11N75)},
	mrnumber = {2488597},
	mrreviewer = {D. R. Heath-Brown},
	number = {3},
	pages = {681-714},
	title = {Effective equidistribution of eigenvalues of {H}ecke operators},
	url = {http://dx.doi.org/10.1016/j.jnt.2008.10.010},
	volume = {129},
	year = {2009}}

@article{Serre,
	author = {Serre, J-P.},
	doi = {10.1090/S0894-0347-97-00220-8},
	fjournal = {Journal of the American Mathematical Society},
	issn = {0894-0347},
	journal = {J. Amer. Math. Soc.},
	mrclass = {11F30 (11F25 11G20 11N37 11R45)},
	mrnumber = {1396897},
	mrreviewer = {Glenn Stevens},
	number = {1},
	pages = {75--102},
	title = {R\'epartition asymptotique des valeurs propres de l'op\'erateur de {H}ecke {$T_p$}},
	url = {http://dx.doi.org/10.1090/S0894-0347-97-00220-8},
	volume = {10},
	year = {1997}}

\end{document}